%% file: KhSurvey.tex
\documentclass[reqno]{amsart}
\usepackage{graphicx}
\usepackage{color}
\usepackage{KhSurvey-style}
\pdfoutput=1

\let\oldhline\hline
\input KhoHo-tables

\def\tCalA{\widetilde\CalA}
\def\slf{\mathfrak{sl}}

\def\KhoHo{{\tt KhoHo}\space}
\def\tbb{\overline{tb}}

\def\hw{\protect\operatorname{hw}}
\def\ohw{\protect\operatorname{ohw}}
\def\thw{\widetilde{\protect\operatorname{hw}}}
\def\tohw{\widetilde{\protect\operatorname{ohw}}}
\let\lra\longrightarrow
\def\odd{\mathrm{odd}}

\theoremstyle{OVplain}
\newtheorem{knightmove}[thm]{Knight-Move Conjecture}

\theoremstyle{OVdefinition}
\newtheorem{example}[thm]{Example}

\begin{document}
\title[Khovanov homology theories and their applications]
{Khovanov homology theories\\ and their applications}
\author[A.~Shumakovitch]{Alexander Shumakovitch}
\address{Department of Mathematics, The George Washington University,
Monroe Hall 2115 G St. NW, Washington, DC 20052, U.S.A.}
\email{Shurik@gwu.edu}
\thanks{The author is partially supported by NSF grant DMS--0707526}
\begin{abstract}
This is an expository paper discussing various versions of Khovanov
homology theories, interrelations between them, their properties, and
their applications to other areas of knot theory and low-dimensional topology.
\end{abstract}
\dedicatory{To my teacher and advisor, Oleg Yanovich Viro,\\
on the occasion of his $60^{th}$ birthday}
\maketitle

\stepcounter{footnote}

\section{Introduction}
Khovanov homology is a special case of {\em categorification}, a novel
approach to construction of knot (or link) invariants that is being actively
developed over
the last decade after a seminal paper~\cite{Kh-Jones} by Mikhail Khovanov. The
idea of categorification is to replace a known polynomial knot (or link)
invariant with a family of chain complexes, such that the coefficients of the
original polynomial are the Euler characteristics of these complexes. Although
the chain complexes themselves depend heavily on a diagram that represents the
link, their homology depend on the isotopy class of the link
only. Khovanov homology categorifies the Jones polynomial~\cite{Jones}.

More specifically, let $L$ be an oriented link in $\R^3$ represented by a
planar diagram $D$ and let $J_L(q)$ be a version of the Jones polynomial of
$L$ that satisfies the following identities (called the {\em Jones skein
relation} and {\em normalization}):
\begin{equation}\label{eq:Jones-skein}
-q^{-2}J_{\includegraphics[scale=0.45]{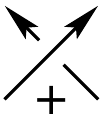}}(q)
+q^2J_{\includegraphics[scale=0.45]{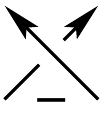}}(q)
=(q-1/q)J_{\includegraphics[scale=0.45]{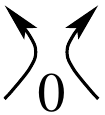}}(q);
\qquad
J_{\includegraphics[scale=0.45]{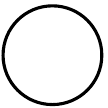}}(q)=q+1/q.
\end{equation}
The skein relation should be understood as relating the Jones polynomials of
three links whose planar diagrams are identical everywhere except in a small
disk, where they are different as depicted in~\eqref{eq:Jones-skein}. The
normalization fixes the value of the Jones polynomial on the trivial knot.
$J_L(q)$ is a Laurent polynomial in $q$ for every link $L$ and is
completely determined by its skein relation and normalization.

In~\cite{Kh-Jones} Mikhail Khovanov assigned to $D$ a family of Abelian groups
$\CalH^{i,j}(L)$ whose isomorphism classes depend on the isotopy class of $L$
only. These groups are defined as homology groups of an appropriate (graded)
chain complex $\CalC^{i,j}(D)$ with integer coefficients.
Groups $\CalH^{i,j}(L)$ are
nontrivial for finitely many values of the pair $(i,j)$ only. The gist of the
categorification is that the graded Euler characteristic of the Khovanov chain
complex equals $J_L(q)$:
\begin{equation}\label{eq:KhEuler-Jones}
J_L(q)=\sum_{i,j}(-1)^iq^jh^{i,j}(L),
\end{equation}
where $h^{i,j}(L)=\rk(\CalH^{i,j}(L))$, the Betti numbers of $\CalH$.
The reader is referred to Section~\ref{sec:Khovanov} for detailed
treatment (see also~\cite{BN-first,Kh-Jones}).

In our paper we also make use of another version of the Jones polynomial,
denoted $\tJ_L(q)$, that satisfies the same skein
relation~\eqref{eq:Jones-skein} but is normalized to equal~$1$ on the trivial
knot. For the sake of completeness, we also list the skein relation for the
original Jones polynomial, $V_L(t)$, from~\cite{Jones}:
\begin{equation}\label{eq:Jones-orig-skein}
t^{-1}V_{\includegraphics[scale=0.45]{pos_Xing-black}}(t)
-tV_{\includegraphics[scale=0.45]{neg_Xing-black}}(t)
=(t^{1/2}-t^{-1/2})V_{\includegraphics[scale=0.45]{smooth_Xing-black}}(t);
\qquad
V_{\includegraphics[scale=0.45]{circle-black}}(t)=1.
\end{equation}
We note that $J_L(q)\in\Z[q,q^{-1}]$ while $V_L(t)\in\Z[t^{1/2},t^{-1/2}]$.
In fact, the terms of $V_L(t)$ have half-integer (resp. integer) exponents if
$L$ has even (resp. odd) number of components. This is one of the main
motivations for our
convention~\eqref{eq:Jones-skein} to be different
from~\eqref{eq:Jones-orig-skein}. We also want to ensure that the Jones
polynomial of the trivial link has only positive coefficients. The different
versions of the Jones polynomial are related as follows:
\begin{equation}\label{eq:Jones-Jones}
J_L(q)=(q+1/q)\tJ_L(q),\qquad \tJ_L(-t^{1/2})=V_L(t),\qquad
V_L(q^2)=\tJ_L(q)
\end{equation}

Another way to look at the Khovanov's identity~\eqref{eq:KhEuler-Jones} is 
via the {\em Poincar\'e polynomial} of the Khovanov homology:
\begin{equation}\label{eq:Kh-Poincare}
Kh_L(t,q)=\sum_{i,j}t^iq^jh^{i,j}(L).
\end{equation}
With this notation, we get
\begin{equation}\label{eq:KhPol-Jones}
J_L(q)=Kh_L(-1,q).
\end{equation}

\begin{figure}
\centerline{$\vcenter{\hbox{\includegraphics[scale=0.66]{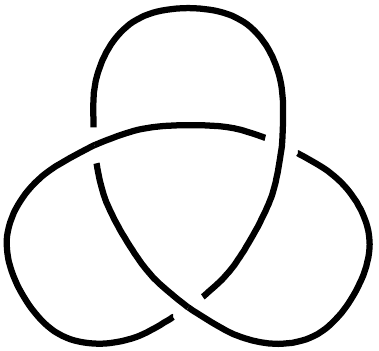}}}
\qquad\qquad\vcenter{\include{trefoil-Kh}}$}
\caption{Right trefoil and its Khovanov homology}\label{fig:trefoil-Kh}
\end{figure}
\begin{example}\label{ex:trefoil}
Consider the right trefoil $K$. Its non-zero homology groups are tabulated
in Figure~\ref{fig:trefoil-Kh}, where the $i$-grading is represented
horizontally and the $j$-grading vertically. The homology is non-trivial for
odd $j$-grading only and, hence, even rows are not shown in the table.
A table entry of $\mathbf1$ or $\mathbf{1_2}$ means that the corresponding group
is $\Z$ or $\Z_2$, respectively (one can find a more interesting example in
Figure~\ref{fig:diff-thickness}). In general, an entry of the form
$\mathbf{a,b_2}$ would correspond to the group $\Z^a\oplus\Z_2^b$.
For the trefoil $K$, we have that $\CalH^{0,1}(K)\simeq\CalH^{0,3}(K)\simeq
\CalH^{2,5}(K)\simeq\CalH^{3,9}(K)\simeq\Z$ and $\CalH^{3,7}(K)\simeq\Z_2$.
Therefore, $Kh_K(t,q)=q+q^3+t^2q^5+t^3q^9$. On the other hand, the Jones
polynomial of $K$ equals $V_K(t)=t+t^3-t^4$. Relation~\eqref{eq:Jones-Jones}
implies that $J_K(q)=(q+1/q)(q^2+q^6-q^8)=q+q^3+q^5-q^9=Kh_K(-1,q)$.
\end{example}

Without going into details, we note that the initial categorification of the
Jones polynomial by Khovanov was followed with a flurry of activity.
Categorifications of the colored Jones
polynomial~\cite{Kh-colored,Beliakova-colored} and skein $\slf(3)$
polynomial~\cite{Kh-sl3} were based on the original Khovanov's construction.
Matrix factorization technique was used to categorify the 
$\slf(n)$ skein polynomials~\cite{KR-sln}, HOMFLY-PT
polynomial~\cite{KR-HOMFLY}, Kauffman polynomial~\cite{KR-SO2N}, and, more
recently, colored $\slf(n)$ polynomials~\cite{Wu-slN,Yonezawa}.
Ozsv\'ath, Szab\'o and, independently, Rasmussen used a completely different
method of Floer homology to categorify the Alexander
polynomial~\cite{OS-knots,Jake-Floer}. Ideas of categorification were
successfully applied to tangles, virtual links, skein modules, and
polynomial invariants of graphs.

One of the most important recent development in the Khovanov homology theory
is the introduction in~2007 of its {\em odd} version by Ozsv\'ath, Rasmussen
and Szab\'o~\cite{Khovanov-odd}. The odd Khovanov homology equals the
original (even) one modulo $2$ and, in particular, categorifies the same Jones
polynomial. On the other hand, the odd and even homology theories often have
drastically different properties (see Sections~\ref{sec:Kh-odd}
and~\ref{sec:Kh-prop} for details). The odd Khovanov homology appears to be
one of the connecting links between Khovanov and Heegaard-Floer homology
theories~\cite{OS-spectral}.

The importance of the Khovanov homology became apparent after a seminal result
by Jacob Rasmussen~\cite{Jake-Milnor}, who used the Khovanov chain complex to
give the first purely combinatorial proof of the Milnor conjecture. This
conjecture states that the $4$-dimensional (slice) genus (and, hence, the
genus) of a $(p,q)$-torus knot equals $\frac{(p-1)(q-1)}2$. It was originally
proved by Kronheimer and Mrowka~\cite{Kronheimer-Mrowka} using the gauge
theory in 1993.

There are numerous other applications of Khovanov homology theories.
They can be used to provide combinatorial proofs of the Slice-Bennequin
Inequality and give upper bounds on the Thurston-Bennequin number of
Legendrian links, detect quasi-alternating links and find topologically
locally-flatly slice knots that are not smoothly slice. We refer the reader to
Section~\ref{sec:Kh-appl} for details.

The goal of this paper is to give an overview of the current state of research
in Khovanov homology. The exposition is mostly self-contained and no advanced
knowledge of the subject is required from the reader. We intentionally limit
the scope of our paper to the categorifications of the Jones polynomial only,
so as to keep its size under control. The reader is referred to other
expository papers on the subject~\cite{Kh-Asaeda,Kh-ICM,Jake-comparison} to
learn more about the interrelations between different types of
categorifications.

We also pay significant attention to experimental aspects of the Khovanov
homology.
As is often the case with new theories, the initial discovery is led by
experiments. It is especially true for Khovanov homology, since it can be
computed by hands for a very limited family of knots only. At the moment,
there are two programs~\cite{katlas-program,Sh-KhoHo} that compute Khovanov
homology. The first one was written by Dror Bar-Natan and his student Jeremy
Green in~2005 and implements the methods from~\cite{BN-fast}. It works
significantly faster for knots with sufficiently many crossings (say, more
than $15$) than the older program \KhoHo by the author. On the other hand,
\KhoHo can compute all the versions of the Khovanov homology that are
mentioned in this paper. It is currently the only program that can deal with
the odd Khovanov homology. Most of the experimental results that are referred
to in this paper were obtained with \KhoHo.

This paper is organized as follows. In Section~\ref{sec:Khovanov} we give a
quick overview of constructions involved in the definition of various
Khovanov homology theories. We compare these theories with each other and 
list their basic properties in Section~\ref{sec:Kh-prop}. 
Section~\ref{sec:Kh-appl} is devoted to some of the more important
applications of the Khovanov homology to other areas of low-dimensional
topology.

This paper was originally presented at the Marcus Wallenberg
Symposium on Perspectives in Analysis, Geometry, and Topology at Stockholm
University in May of 2008. The author would like to thank all the organizers
of the Symposium for a very successful and productive meeting. He extends his
special thanks to Ilia Itenberg, Burglind J\"oricke, and Mikael Passare,
the editors of these Proceedings, for their patience with the author.
The author is indebted to Mikhail Khovanov for many advises and enlightening
discussions during the work on this paper. Finally, the author would like to
express his deepest gratitude to Oleg Yanovich Viro for introducing him to the
wonderful world of topology $20$ years ago and for continuing to be his guide
in this world ever since.

\section{Definition of the Khovanov homology}\label{sec:Khovanov}

In this section we give a brief outline of various Khovanov homology theories
starting with the original Khovanov's construction. Our setting is slightly
more general than the one in the Introduction as we allow different
coefficient rings, not only $\Z$.

\subsection{Algebraic preliminaries}
Let $R$ be a commutative ring with unity. In this paper, we are mainly
interested in the cases when $R=\Z$, $\Q$, or $\Z_2$.

\begin{defin}\label{def:graded-module}
A {\sl $\Z$-graded} (or simply {\sl graded}) $R$-module $M$ is an $R$-module
decomposed into a direct sum $M=\bigoplus_{j\in\Z}M_j$, where each $M_j$ is an
$R$-module itself. The summands $M_j$ are called {\em homogeneous components}
of $M$ and elements of $M_j$ are called the {\em homogeneous elements of
degree $j$}.
\end{defin}

\begin{defin}\label{def:graded-dim}
Let $M=\bigoplus_{j\in\Z}M_j$ be a graded free $R$-module. The {\em graded
dimension} of $M$ is the power series $\dim_q(M)=\sum_{j\in\Z}q^j\dim(M_j)$ in
variable $q$. If $k\in\Z$, the {\em shifted module} $M\{k\}$ is defined as
having homogeneous components $M\{k\}_j=M_{j-k}$.
\end{defin}

\begin{defin}\label{def:graded-map}
Let $M$ and $N$ be two graded $R$-modules. A map $\Gf:M\to N$ is said to be
{\em graded of degree $k$} if $\Gf(M_j)\subset N_{j+k}$ for each $j\in\Z$.
\end{defin}

\begin{attn}\label{attn:grading-props}
It is an easy exercise to check that $\dim_q(M\{k\})=q^k\dim_q(M)$,
$\dim_q(M\oplus N)=\dim_q(M)+\dim_q(N)$, and
$\dim_q(M\otimes_R N)=\dim_q(M)\dim_q(N)$, where $M$ and $N$ are graded
$R$-modules. Moreover, if $\Gf:M\to N$ is a graded map of degree $k'$, then
the {\sl shifted map} $\Gf:M\to N\{k\}$ is graded of degree $k'+k$. We
slightly abuse the notation here by denoting the shifted map in the same way
as the map itself.
\end{attn}

\begin{defin}\label{def:graded-Euler}
Let $(\CalC,d)=\cdots\lra\CalC^{i-1}
\stackrel{d^{i{-}1}}{\lra}\CalC^i
\stackrel{d^i}{\lra}\CalC^{i+1}\lra\cdots$ be a
(co)chain complex of graded free $R$-modules with graded
differentials $d^i$ having
degree $0$ for all $i\in\Z$. Then the {\sl graded Euler characteristic}
of $\CalC$ is defined as $\chi_q(\CalC)=\sum_{i\in\Z}(-1)^i\dim_q(\CalC^i)$.
\end{defin}

\begin{rem}
One can think of a graded (co)chain complex of $R$-modules as a {\em bigraded
$R$-module} where the homogeneous components are indexed by pairs of numbers
$(i,j)\in\Z^2$.
\end{rem}

Let $A=R[X]/X^2$ be the algebra of truncated polynomials. As an $R$-module,
$A$ is freely generated by $1$ and $X$. We put grading on $A$ by specifying
that $\deg(1)=1$ and $\deg(X)=-1\footnote{We follow the original grading
convention from~\cite{Kh-Jones} and~\cite{BN-first} here. It is different
by a sign from the one in~\cite{Kh-Asaeda}.}$. In other words, $A\simeq
R\{1\}\oplus R\{-1\}$ and $\dim_q(A)=q+q^{-1}$. At the same time, $A$ is a
(graded) commutative algebra with the unity $1$ and multiplication $m:A\otimes
A\to A$ given by \begin{equation}\label{eq:A-mult}
m(1\otimes 1)=1,\qquad m(1\otimes X)=m(X\otimes1)=X,\qquad m(X\otimes X)=0.
\end{equation}

$A$ can also be equipped with a coalgebra structure with comultiplication
$\Delta:A\to A\otimes A$ and counit $\Ge:A\to R$ defined as
\begin{align}\label{eq:A-comult}
\GD(1)&=1\otimes X+X\otimes 1,& \GD(X)&=X\otimes X;\\
\Ge(1)&=0,&\Ge(X)&=1.
\end{align}
The comultiplication $\GD$ is coassociative and cocommutative and satisfies
\begin{align}
(m\otimes\id_A)\circ(\id_A\otimes\GD)&=\GD\circ m\\
(\Ge\otimes\id_A)\circ\GD&=\id_A
\end{align}
Together with the unit map $\Gi:R\to A$ given by $\Gi(1)=1$, this makes $A$
into a commutative Frobenius algebra over $R$~\cite{Kh-Frobenius}.

It follows directly from the definitions that $\Gi$, $\Ge$, $m$, and $\GD$ are
graded maps with
\begin{equation}
\label{eq:cobord-grading}
\deg(\Gi)=\deg(\Ge)=1\quad\hbox{ and }\quad \deg(m)=\deg(\GD)=-1.
\end{equation}

\subsection{Khovanov chain complex}\label{sec:Kh-complex}
Let $L$ be an oriented link and $D$ its planar diagram. We assign a number
$\pm1$, called {\em sign}, to every crossing of $D$ according to the rule
depicted in Figure~\ref{fig:crossing-signs}. The sum of these signs over all
the crossings of $D$ is called the {\em writhe number} of $D$ and is denoted
by $w(D)$.

\begin{figure}
\captionindent 0.35\captionindent
\begin{minipage}{1.8in}
\centerline{\input{Xing_signs.pspdftex}}
\caption{Positive and negative crossings}
\label{fig:crossing-signs}
\end{minipage}
\hfill
\begin{minipage}{3.1in}
\centerline{\input{markers.pspdftex}}
\caption{Positive and negative markers and the corresponding resolutions of a
diagram.}
\label{fig:markers}
\end{minipage}
\end{figure}

Every crossing of $D$ can be {\em resolved} in two different ways according to
a choice of a {\em marker}, which can be either {\em positive} or {\em
negative}, at this crossing (see Figure~\ref{fig:markers}). A collection of
markers chosen at every crossing of a diagram $D$ is called a {\em (Kauffman)
state} of $D$. For a diagram with $n$ crossings, there are, obviously, $2^n$
different states. Denote by $\Gs(s)$ the difference between the numbers of
positive and negative markers in a given state $s$. Define
\begin{equation}\label{eq:state-ij}
i(s)=\frac{w(D)-\Gs(s)}2,\qquad j(s)=\frac{3w(D)-\Gs(s)}2.
\end{equation}
Since both $w(D)$ and $\Gs(s)$ are congruent to $n$ modulo $2$, $i(s)$ and
$j(s)$ are always integer. For a given state $s$, the result of
the resolution of $D$ at each crossing according to $s$ is a family $D_s$ of
disjointly embedded circles. Denote the number of these circles by $|D_s|$.

For each state $s$ of $D$, let $\CalA(s)=A^{\otimes|D_s|}\{j(s)\}$. One
should understand this construction as assigning a copy of algebra $A$ to each
circle from $D_s$, taking the tensor product of all of these copies, and
shifting the grading of the result by $j(s)$. By construction, $\CalA(s)$ is
a graded free $R$-module of graded dimension 
$\dim_q(\CalA(s))=q^{j(s)}(q+q^{-1})^{|D_s|}$. Let
$\CalC^i(D)=\bigoplus_{i(s)=i}\CalA(s)$ for each $i\in\Z$. In order to make
$\CalC(D)$ into a graded complex, we need to define a (graded) differential
$d^i:\CalC^i(D)\to\CalC^{i+1}(D)$ of degree $0$. But even before this
differential is defined, the (graded) Euler characteristic of $\CalC(D)$
makes sense.

\begin{lem}\label{lem:Euler-Jones}
The graded Euler characteristic of $\CalC(D)$ equals the Jones polynomial of
the link $L$. That is, $\chi_q(\CalC(D))=J_L(q)$.
\end{lem}
\begin{proof}
{\abovedisplayskip-\baselineskip
\[\begin{split}
\chi_q(\CalC(D))&=\sum_{i\in\Z}(-1)^i\dim_q(\CalC^i(D))\\
&=\sum_{i\in\Z}(-1)^i\sum_{i(s)=i}\dim_q(\CalA(s))\\
&=\sum_s(-1)^{i(s)}q^{j(s)}(q+q^{-1})^{|D_s|}\\
&=\sum_s(-1)^{\frac{w(D)-\Gs(s)}2}q^{\frac{3w(D)-\Gs(s)}2}(q+q^{-1})^{|D_s|}.
\end{split}\]}

Let us forget for a moment that $A$ denotes an algebra and (temporarily) use
this letter for a variable. Substituting $(-A^{-2})$ instead of $q$ and
noticing that $w(D)\equiv\Gs(s)\pmod2$, we arrive at
\begin{equation*}
\chi_q(\CalC(D))=(-A)^{-3w(D)}\sum_sA^{\Gs(s)}(-A^2{-}A^{-2})^{|D_s|}
=(-A^2{-}A^{-2})\langle L\rangle_N,
\end{equation*}
where $\langle L\rangle_N$ is the normalized Kauffman
bracket polynomial of $L$ (see~\cite{Kauffman-bracket} for details). The
normalized bracket polynomial of a link is related to the bracket polynomial
of its diagram as $\langle L\rangle_N=(-A)^{-3w(D)}\langle D\rangle$. Kauffman
proved in~\cite{Kauffman-bracket} that $\langle L\rangle_N$ equals the Jones
polynomial $V_L(t)$ of $L$ after substituting $t^{-1/4}$ instead of $A$. The
relation~\eqref{eq:Jones-Jones} between $V_L(t)$ and $J_L(q)$ completes our
proof.
\end{proof}

\begin{figure}
\centerline{\vbox{\halign{#\hfill\cr\hskip-0.5em\input{differential.pspdftex}\cr
\noalign{\vskip -1.2\baselineskip}
\hbox{\vbox to0pt{\vss\halign{#\hfill\cr
$m(1{\otimes}1){=}1$,\enspace $m(1{\otimes}X){=}m(X{\otimes}1){=}X$,\enspace
$m(X{\otimes}X){=}0$\cr
\noalign{\medskip}
$\GD(1)=1\otimes X+X\otimes 1,\qquad\GD(X)=X\otimes X$\cr
\noalign{\smallskip}}}}\cr
}}}
\caption{Diagram resolutions corresponding to adjacent states and 
maps between the algebras assigned to the circles}\label{fig:res-change}
\end{figure}
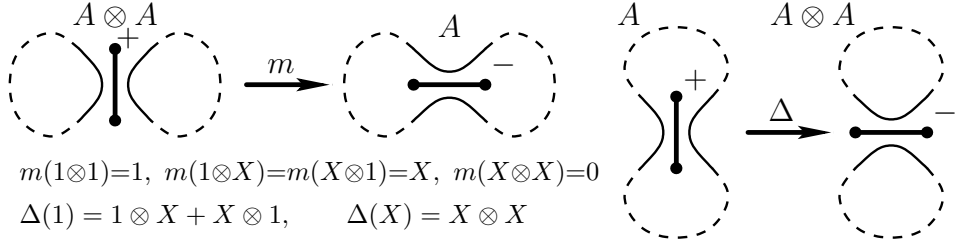

Let $s_+$ and $s_-$ be two states of $D$ that differ at a single
crossing, where $s_+$ has a positive marker while $s_-$ has a negative one.
We call two such states {\em adjacent}. In this case, $\Gs(s_-)=\Gs(s_+)-2$
and, consequently, $i(s_-)=i(s_+)+1$ and $j(s_-)=j(s_+)+1$. Consider now the
resolutions of $D$ corresponding to $s_+$ and $s_-$. One can
readily see that $D_{s_-}$ is obtained from $D_{s_+}$ by either merging two
circles into one or splitting one circle into two (see
Figure~\ref{fig:res-change}). All the circles that do not pass through the
crossing at which $s_+$ and $s_-$ differ, remain unchanged. We define 
$d_{s_+:s_-}:\CalA(s_+)\to\CalA(s_-)$ as either $m\otimes\id$ or
$\GD\otimes\id$ depending on whether the circles merge on split. Here, the
multiplication or comultiplication is performed on the copies of $A$ that are
assigned to the affected circles, as on Figure~\ref{fig:res-change}, while
$d_{s_+:s_-}$ acts as identity on all the $A$'s corresponding to the
unaffected ones. The difference in grading shift between $\CalA(s_+)$ and
$\CalA(s_-)$ and \eqref{eq:cobord-grading} 
ensure that $\deg(d_{s_+:s_-})=0$ by~\ref{attn:grading-props}.

We need one more ingredient in order to finish the definition of the
differential on $\CalC(D$), namely, an ordering of the crossings of $D$. 
For an adjacent pair of states $(s_+,s_-)$, define $\Gx(s_+,s_-)$ to be the 
number of the {\em negative} markers in $s_+$ (or $s_-$) that appear in the
ordering of the crossings {\em after} the crossing at which $s_+$ and $s_-$
differ. Finally, let $d^i=\sum_{(s_+,s_-)}(-1)^{\Gx(s_+,s_-)}d_{s_+:s_-}$,
where $(s_+,s_-)$ runs over all adjacent pairs of states with $i(s_+)=i$.
It is straightforward to verify~\cite{Kh-Jones} that $d^{i+1}\circ d^i=0$ and,
hence, $d:\CalC(D)\to\CalC(D)$ is indeed a differential.

\begin{defin}[Khovanov,~\cite{Kh-Jones}]\label{def:Khovanov}
The resulting (co)chain complex
$\CalC(D)=\cdots\lra\CalC^{i-1}(D)
\stackrel{d^{i{-}1}}{\lra}\CalC^i(D)
\stackrel{d^i}{\lra}\CalC^{i+1}(D)\lra\cdots$
is called the {\em Khovanov chain complex} of the diagram $D$.
The homology of $\CalC(D)$ with respect to $d$ is called the {\em Khovanov
homology} of $L$ and is denoted by $\CalH(L)$. We write $\CalC(D;R)$ and
$\CalH(L;R)$ if we want to emphasize the ring of coefficients that we work
with. If $R$ is omitted from the notation, integer coefficients are assumed.
\end{defin}

\begin{thm}[Khovanov,~\cite{Kh-Jones}, see also~\cite{BN-first}]
\label{thm:Khovanov}
The isomorphism class of $\CalH(L;R)$ depends on the isotopy class of $L$
only and, hence, is a link invariant. In particular, it does not depend on the
ordering chosen for the crossings of $D$. $\CalH(L;R)$ categorifies $J_L(q)$, a
version of the Jones polynomial defined by~\eqref{eq:Jones-skein}.
\end{thm}

\begin{rem} One can think of $\CalC(D;R)$ as a bigraded (co)chain complex
$\CalC^{i,j}(D;R)$ with a differential of bidegree $(1,0)$. In this case, $i$
is the homological grading of this complex, and $j$ is its $q$-grading, also
called the {\em Jones grading}. Correspondingly, $\CalH(L;R)$ can be considered
to be a bigraded $R$-module as well.
\end{rem}

\begin{attn}\label{attn:grading-parity}
Let $\#L$ be the number of components of a link $L$. One can check that
$j(s)+|D_s|$ is congruent modulo $2$ to $\#L$ for every state $s$. It follows
that $\CalC(D;R)$ has non-trivial homogeneous components only in the degrees
that have the same parity as $\#L$. Consequently, $\CalH(L;R)$ is non-trivial
only in the $q$-gradings with this parity (see Example~\ref{ex:trefoil}).
\end{attn}

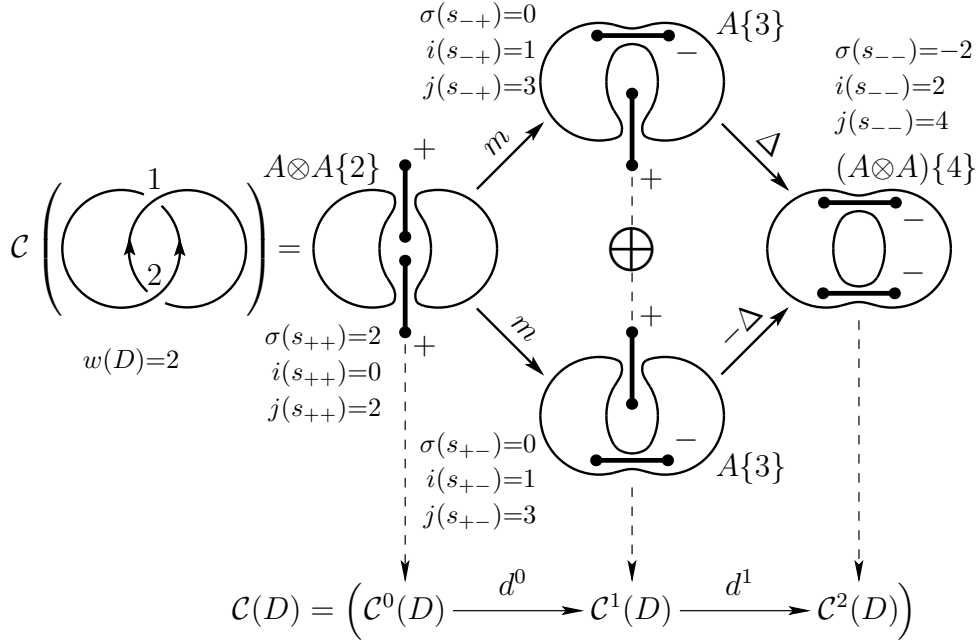
\begin{figure}
\centerline{\input{HopfLink-complex.pspdftex}}
\caption{Khovanov chain complex for the Hopf link}\label{fig:Kh-Hopf}
\end{figure}

\begin{example}\label{ex:Hopf-example}
Figure~\ref{fig:Kh-Hopf} shows the Khovanov chain complex for the Hopf link
with the indicated orientation. The diagram has two positive crossings, so its
writhe number is $2$. Let $s_{\pm\pm}$ be the four possible resolutions of
this diagram, where each ``$+$'' or ``$-$'' describes the sign of the marker
at the corresponding crossings. The chosen ordering of crossings is depicted by
numbers placed next to them. By looking at Figure~\ref{fig:Kh-Hopf}, one
easily computes that
$\CalA(s_{++})=A^{\otimes2}\{2\}$, $\CalA(s_{+-})=\CalA(s_{-+})=A\{3\}$, and 
$\CalA(s_{--})=A^{\otimes2}\{4\}$. Correspondingly,
$\CalC^0(D)=\CalA(s_{++})=A^{\otimes2}\{2\}$, $\CalC^1(D)=\CalA(s_{+-})\oplus
\CalA(s_{-+})=(A\oplus A)\{3\}$, and
$\CalC^2(D)=\CalA(s_{--})=A^{\otimes2}\{4\}$.
It is convenient to arrange the four resolutions in the corners of a square
placed in the plane in such a way that its diagonal from $s_{++}$ to $s_{--}$
is horizontal. Then the edges of this square correspond to the maps between
the adjacent states (see Figure~\ref{fig:Kh-Hopf}). We notice that only one of
these maps, namely the one corresponding to the edge from $s_{+-}$ to $s_{--}$,
comes with the negative sign.
\end{example}

In general, $2^n$ resolutions of a diagram $D$ with $n$ crossings can be
arranged into an $n$-dimensional {\em cube of resolutions}, where vertices
correspond to the $2^n$ states of~$D$. The edges of this cube connect adjacent
pairs of states and can be oriented from $s_+$ to $s_-$. Every edge is assigned
either $m$ or $\GD$ with the sign $(-1)^{\Gx(s_+,s_-)}$, as described above.
It is easy to check that this makes each square (that is, a $2$-dimensional
face) of the cube anti-commutative (all squares are commutative without the
signs). Finally, the differential $d^i$ restricted to each summand $\CalA(s)$
with $i(s)=i$ equals the sum of all the maps assigned to the edges that
originate at~$s$.

\subsection{Reduced Khovanov homology}\label{sec:Kh-reduced}
Let, as before, $D$ be a diagram of an oriented link $L$. Fix a base point on
$D$ that is different from all the crossings. For each state $s$, we define
$\tCalA(s)$ in almost the same way as $\CalA(s)$, except that we assign $XA$
instead of $A$ to the circle from the resolution $D_s$ of $D$ that contains
that base point. That is,
$\tCalA(s)=\left((XA)\otimes A^{\otimes(|D_s|-1)}\right)\{j(s)\}$.
We can now build the {\em reduced Khovanov chain complex} $\tCalC(D;R)$ in
exactly the same way as $\CalC(D;R)$ by replacing $\CalA$ with $\tCalA$
everywhere. The grading shifts and differentials remain the same. It is easy
to see that $\tCalC(D;R)$ is a subcomplex of $\CalC(D;R)$ of index $2$. In
fact, it is the image of the chain map $\CalC(D;R)\to\CalC(D;R)$ that acts
by multiplying elements assigned to the circle containing the base point by
$X$.

\begin{defin}[Khovanov~\cite{Kh-patterns}, cf.~\ref{def:Khovanov}]
\label{def:Kh-reduced}
The homology of $\tCalC(D;R)$ is called the {\em reduced Khovanov Homology of
$L$} and is denoted by $\tCalH(L;R)$. It is clear from the construction of
$\tCalC(D;R)$ that its graded Euler characteristic equals $\tJ_L(q)$.
\end{defin}

\begin{thm}[Khovanov~\cite{Kh-patterns}, cf.~\ref{thm:Khovanov}]
\label{thm:Kh-reduced}
The isomorphism class of $\tCalH(L;R)$ is a link invariant that categorifies
$\tJ_L(q)$, a version of the Jones polynomial defined by~\eqref{eq:Jones-skein}
and~\eqref{eq:Jones-Jones}. Moreover, if two base points are chosen on the
same component of $L$, then the corresponding reduced Khovanov homologies are
isomorphic. On the other hand, $\tCalH(L;R)$ might depend on the component
of $L$ that the base point is chosen on.
\end{thm}

Although $\tCalC(D;R)$ can be determined from $\CalC(D;R)$, it is in general
not clear how $\CalH(L;R)$ and $\tCalH(L;R)$ are related. There are several
examples of pairs of knots (the first one being $14^n_{9933}$ and
$\overline{15}^n_{129763}$\footnote{Here, $14^n_{9933}$ denotes the
non-alternating knot number 9933 with 14 crossings from the Knotscape knot
table~\cite{Knotscape} and $\overline{15}^n_{129763}$ is the mirror image of
the knot $15^n_{129763}$. See also remark on page~\pageref{rem:knots-enum}.})
that have the same rational Khovanov homology, but different rational reduced
Khovanov homology. No such examples are known for homologies over $\Z$ among
all prime knots with at most $15$ crossings. On the other hand, it is proved
that $\CalH(L;\Z_2)$ and $\tCalH(L;\Z_2)$ determine each other completely.

\begin{thm}[\cite{Sh-torsion}]\label{thm:me-Z2}
$\CalH(L;\Z_2)\simeq\tCalH(L;\Z_2)\otimes_{\Z_2}A_{\Z_2}$. In particular,
$\tCalH(L;\Z_2)$ does not depend on the component that the base point is
chosen on.
\end{thm}

\begin{rem}
$XA\simeq R\{0\}$ as a graded $R$-module. It follows that $\tCalC$ and
$\tCalH$ are non-trivial only in the $q$-gradings with parity different
from that of $\#L$, the number of components of $L$
(cf.~\ref{attn:grading-parity}).
\end{rem}

\subsection{Odd Khovanov homology}\label{sec:Kh-odd}
In 2007, Ozsv\'ath, Rasmussen and Szab\'o introduced~\cite{Khovanov-odd} an
{\em odd} version of the Khovanov homology. In their theory, the nilpotent
variables $X$ assigned to each circle in the resolutions of the link diagram
(see Section~\ref{sec:Kh-complex}) anti-commute rather than commute. The odd
Khovanov homology equals the original ({\em even}) one modulo $2$ and, in
particular, categorifies the same Jones polynomial. In fact, the corresponding
chain complexes are isomorphic as free bigraded $R$-modules and their
differentials are only different by signs. On the other hand, the resulting
homology theories often have drastically different properties. We define the
odd Khovanov homology below.

Let $L$ be an oriented link and $D$ its planar diagram. To each resolution $s$
of $D$ we assign a free graded $R$-module $\GL(s)$ as follows. Label all
circles from the resolution $D_s$ by some independent variables, say,
$X^s_1,X^s_2,\,\dots\,,X^s_{|D_s|}$ and let
$V_s=V(X^s_1,X^s_2,\,\dots\,,X^s_{|D_s|})$ be
a free $R$-module generated by them. We define $\GL(s)=\GL\!^*(V_s)$, the
exterior algebra of $V_s$. Then $\GL(s)=\GL\!^0(V_s)\oplus\GL\!^1(V_s)\oplus
\cdots\oplus\GL\!^{|D_s|}(V_s)$ and we grade $\GL(s)$ by specifying
$\GL(s)_{|D_s|-2k}=\GL\!^k(V_s)$ for each $0\le k\le|D_s|$, where
$\GL(s)_{|D_s|-2k}$ is the homogeneous component of $\GL(s)$ of degree
$|D_s|-2k$. It is an easy exercise for the reader to check that
$\dim_q(\GL(s))=\dim_q(A^{\otimes|D_s|})$.

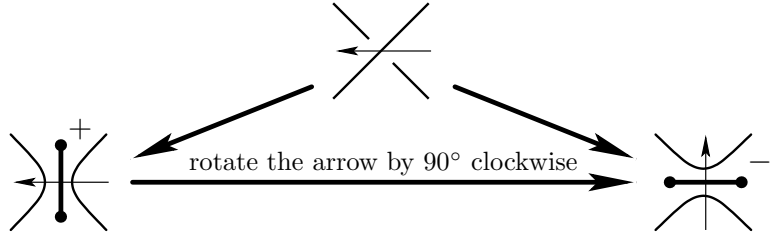
\begin{figure}
\centerline{\input{odd-arrows.pspdftex}}
\caption{Choice of arrows at the diagram crossings}\label{fig:odd-arrows}
\end{figure}

Just as in the case of the even Khovanov homology, these $R$-modules $\GL(s)$
can be arranged into an $n$-dimensional cube of resolutions. Let
$\CalC_{\odd}^i(D)=\bigoplus_{i(s)=i}\GL(s)\{j(s\}$. Then, similarly to
Lemma~\ref{lem:Euler-Jones}, we have that
$\chi_q(\CalC_{\odd}(D))=J_L(q)$. In fact, $\CalC_{\odd}(D)\simeq\CalC(D)$ as
bigraded $R$-modules. In order to
define the differential on $\CalC_{\odd}$, we need to introduce an additional
structure, a choice of an arrow at each crossing of $D$ that is parallel to
the {\em negative} marker at that crossing (see Figure~\ref{fig:odd-arrows}).
There are obviously $2^n$ such choices. For every state $s$ on $D$, we place
arrows that connect two branches of $D_s$ near each (former) crossing
according to the rule from Figure~\ref{fig:odd-arrows}.

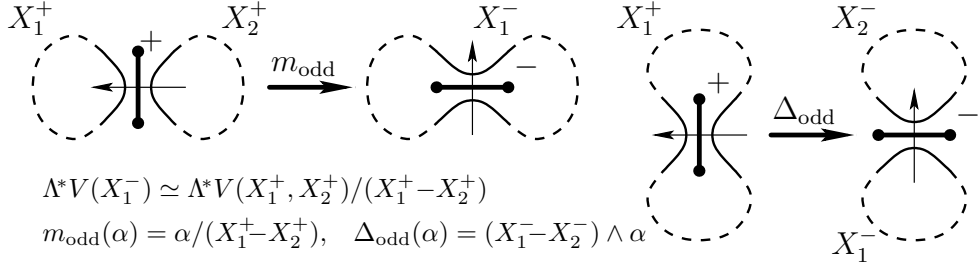
\begin{figure}
\centerline{\vbox{\halign{#\hfill\cr
\hskip-0.5em\input{differential-odd.pspdftex}\cr
\noalign{\vskip -1.2\baselineskip}
\hbox{\vbox to0pt{\vss\halign{#\hfill\cr
$\GL\!^*V(X^-_1)\simeq\GL\!^*V(X^+_1,X^+_2)/(X^+_1{-}X^+_2)$\cr
\noalign{\medskip}
$m_{\odd}(\Ga)=\Ga/(X^+_1\!\!{-}X^+_2)$,\quad
$\GD_{\odd}(\Ga)=(X^-_1\!\!{-}X^-_2)\wedge\Ga$\cr
\noalign{\medskip}}}}\cr
}}}
\caption{Adjacent states and differentials in the odd Khovanov chain complex}
\label{fig:odd-diff}
\end{figure}

We now assign (graded) maps $m_{\odd}$ and $\GD_{\odd}$ to each edge of the cube
of resolutions that connects adjacent states $s_+$ and $s_-$. If $s_-$ is
obtained from $s_+$ by merging two circles together, then
$\GL(s_-)\simeq\GL(s_+)/(X^+_1-X^+_2)$,
where $X^+_1$ and $X^+_2$ are the generators
of $V_{s_+}$ corresponding to the two merging circles, as depicted in
Figure~\ref{fig:odd-diff}. We define $m_{\odd}:\GL(s_+)\to\GL(s_-)$ to be
this isomorphism composed with the projection
$\GL(s_+)\to\GL(s_+)/(X^+_1-X^+_2)$.

The case when one circle splits into two is more interesting. Let $X^-_1$ and
$X^-_2$ be the generators of $V_{s_-}$ corresponding to these two circles such
that the arrow points from $X^-_1$ to $X^-_2$ (see Figure~\ref{fig:odd-diff}).
Now for each generator $X^+_k$ of $V_{s_+}$, we define
$\GD_{\odd}(X^+_k)=(X^-_1-X^-_2)\wedge X^-_{\eta(k)}$ where $\eta$ is the
correspondence between circles in $D_{s_+}$ and $D_{s_-}$. While $\eta(1)$ can
equal either $1$ or $2$, this choice does not affect $\GD_{\odd}(X^+_1)$ since
$(X^-_1-X^-_2)\wedge X^-_2=X^-_1\wedge X^-_2=-X^-_2\wedge X^-_1=
(X^-_1-X^-_2)\wedge X^-_1$.

This definition makes each square in the cube of resolutions either
commutative, or anti-commutative, or both. The latter case means that both
double-composites corresponding to the square are trivial. This is a major
departure from the situation that we had in the even case, where each square
was commutative. In particular, it makes the choice of signs on the edges of
the cube much more involved.

\begin{thm}[Ozsv\'ath--Rasmussen--Szab\'o~\cite{Khovanov-odd}]
It is possible to assign a sign to each edge in this (odd) cube of resolutions
in such a way that every square becomes anti-commutative. This results in a
graded (co)chain complex $\CalC_{\odd}(D;R)$.
The homology $\CalH_{\odd}(L;R)$ of $\CalC_{\odd}(D;R)$ does not depend on the
choice of arrows at the crossings, the choice of edge signs, and some other
choices needed in the construction. Moreover, the isomorphism class of
$\CalH_{\odd}(L;R)$ is a link invariant, called {\em odd Khovanov homology},
that categorifies $J_L(q)$.
\end{thm}

\begin{rem}
There is no explicit construction for assigning signs to the edges of
the cube of resolutions in the case of the odd Khovanov chain complex. The
Theorem above only ensures that signs exist.
\end{rem}

\begin{attn}\label{attn:odd-Z2}
By comparing the definitions of $\CalC_{\odd}(D;\Z_2)$ and $\CalC(D;\Z_2)$, it
is easy to see that they are isomorphic as graded chain complexes (since the
signs do not matter modulo $2$). It follows that
$\CalH_{\odd}(D;\Z_2)\simeq\CalH(D;\Z_2)$ as well.
\end{attn}

\begin{attn}\label{attn:odd-reduced}
One can construct {\em reduced} odd Khovanov chain complex $\tCalC_{\odd}(D;R)$
and {\em reduced} odd Khovanov homology $\tCalH_{\odd}(L;R)$ using methods
similar to those from Sections~\ref{sec:Kh-reduced}. In this case, contrary to
the even situation, reduced and non-reduced odd Khovanov homology determine
each other completely (see~\cite{Khovanov-odd}).
Namely, $\CalH_{\odd}(L;R)\simeq\tCalH_{\odd}(L;R)\{1\}\oplus
\tCalH_{\odd}(L;R)\{-1\}$ (cf. Theorem~\ref{thm:me-Z2}). It is therefore enough
to consider the reduced version of the odd Khovanov homology only.
\end{attn}

\begin{figure}
\centerline{\input{HopfLink-complex-odd.pspdftex}}
\caption{Odd Khovanov chain complex for the Hopf link}\label{fig:Kh-Hopf-odd}
\end{figure}
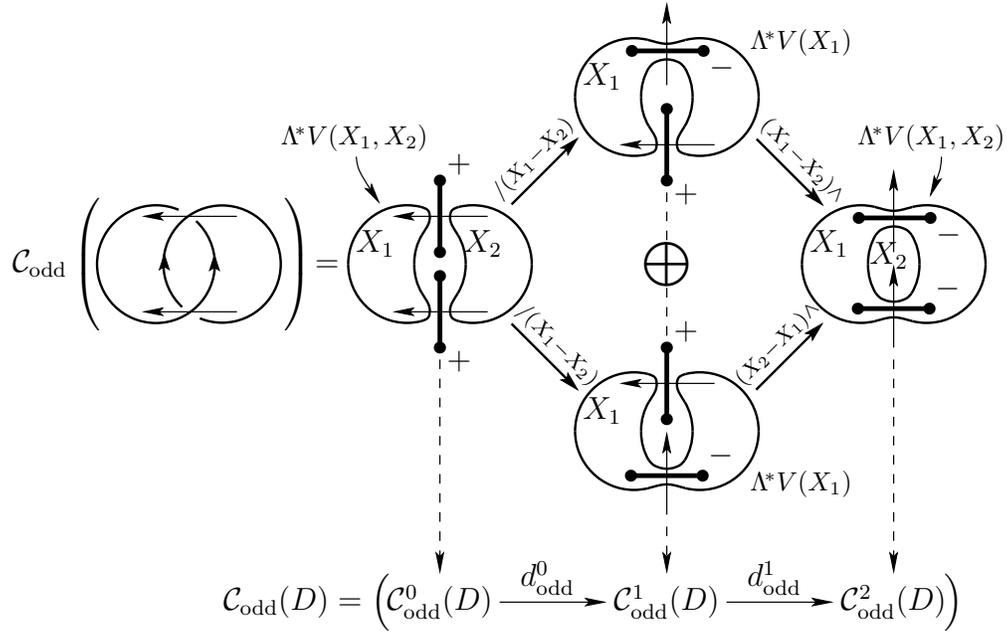

\begin{example}
The odd Khovanov chain complex for the Hopf link is depicted in
Figure~\ref{fig:Kh-Hopf-odd}. All the grading shifts in this case are the same
as in Example~\ref{ex:Hopf-example} and on Figure~\ref{fig:Kh-Hopf}, so we do
not list them again. We notice that the resulting square of resolutions is
anti-commutative, so no adjustment of signs is needed.
\end{example}

The odd Khovanov homology should provide an insight into interrelations
between Khovanov and Heegaard-Floer~\cite{OS-HF} homology theories. Its
definition was motivated by the following result.

\begin{thm}[Ozsv\'ath--Szab\'o~\cite{OS-spectral}]\label{thm:OS-spec}
For each link $L$ with a diagram $D$, there exists a spectral sequence with
$E^1=\tCalC(D;\Z_2)$ and $E^2=\tCalH(L;\Z_2)$ that converges to the
$\Z_2$-Heegaard-Floer homology $\widehat{HF}(\GS(L);\Z_2)$ of the double
branched cover $\GS(L)$ of $S^3$ along $L$.
\end{thm}

\begin{conjec}\label{conj:OS-spec}
There exists a spectral sequence that starts with $\tCalC_{\odd}(D;\Z)$
and $\tCalH_{\odd}(L;\Z)$ and converges to $\widehat{HF}(\GS(L);\Z)$.
\end{conjec}

\section{Properties of the Khovanov homology}\label{sec:Kh-prop}

In this section we summarize main properties of the Khovanov homology and list
related constructions. We emphasize similarities and differences in properties
exhibited by different versions of the Khovanov homology. Some of them were
already mentioned in the previous sections.

\setcounter{thm}{0}
\begin{attn}
Let $L$ be an oriented link and $D$ its planar diagram. Then
\begin{itemize}
\item $\tCalC(D;R)$ is a subcomplex of $\CalC(D;R)$ of index $2$.
\item $\CalH_{\odd}(L;\Z_2)\simeq\CalH(L;\Z_2)$ and
$\tCalH_{\odd}(L;\Z_2)\simeq\tCalH(L;\Z_2)$.
\item $\chi_q(\CalH(L;R))=\chi_q(\CalH_{\odd}(L;R))=J_L(q)$ and\\
$\chi_q(\tCalH(L;R))=\chi_q(\tCalH_{\odd}(L;R))=\tJ_L(q)$.
\item
$\CalH(L;\Z_2)\simeq\tCalH(L;\Z_2)\otimes_{\Z_2}\!A_{\Z_2}$~\cite{Sh-torsion}
and $\CalH_{\odd}(L;R)\simeq\tCalH_{\odd}(L;R)\{1\}\oplus
\tCalH_{\odd}(L;R)\{-1\}$~\cite{Khovanov-odd}. On the other hand,
$\CalH(L;\Z)$ and $\CalH(L;\Q)$ do not split in general.
\item For links, $\tCalH(L;\Z_2)$ and $\tCalH_{\odd}(L;R)$ do not depend on the
choice of a component with the base point. This is, in general, not the case
for $\tCalH(L;\Z)$ and $\tCalH(L;\Q)$.
\item If $L$ is a non-split alternating link, then $\CalH(L;\Q)$,
$\tCalH(L;R)$, and $\tCalH_{\odd}(L;R)$ are completely determined by the
Jones polynomial and signature of~$L$~\cite{Kh-patterns,Lee,Khovanov-odd}.
\item $\CalH(L;\Z_2)$ and $\CalH_{\odd}(L;\Z)$ are invariant under the
component-preserving link mutations~\cite{Bloom-mutation,Wehrli-mutation2}. It
is unclear whether the same holds true for $\CalH(L;\Z)$. On the other hand,
$\CalH(L;\Z)$ is known not to be preserved under a mutation that
exchanges components of a link~\cite{Wehrli-mutation1} and under a
cabled mutation~\cite{DGShT}.
\item $\CalH(L;\Z)$ almost always has torsion (except for several special
cases), but mostly of order $2$. The first knot with $4$-torsion is the
$(4,5)$-torus knot that has 15 crossings. The first known knot with
$3$-torsion is the $(5,6)$-torus knot with 24 crossings. On the other hand,
$\CalH_{\odd}(L;\Z)$ was observed to have torsion of various orders even for
knots with relatively few crossings (see remark on
page~\pageref{rem:odd-torsion}), although orders $2$ and $3$ are the most
popular.
\item $\tCalH(L;\Z)$ has very little torsion. The first knot with torsion has
13 crossings. On the other hand, $\tCalH_{\odd}(L;\Z)$ has as much torsion as
$\CalH_{\odd}(L;\Z)$.
\end{itemize}
\end{attn}

\begin{rem}
The properties above show that $\tCalH_{\odd}(L;\Z)$ behaves
similarly to $\tCalH(L;\Z_2)$ but not to $\tCalH(L;\Z)$. This is by design
(see~\ref{thm:OS-spec} and~\ref{conj:OS-spec}).
\end{rem}

\subsection{Homological thickness}
\begin{defin}\label{def:hw}
Let $L$ be a link. The {\em homological width} of $L$ over a ring $R$ is the
minimal number of adjacent diagonals $j-2i=const$ such that $\CalH(L;R)$ is
zero outside of these diagonals.
It is denoted by $\hw_R(L)$. The {\em reduced homological width}, $\thw_R(L)$
of $L$, {\em odd homological width}, $\ohw_R(L)$ of $L$, and {\em reduced odd
homological width}, $\tohw_R(L)$ of $L$ are defined similarly.
\end{defin}

\begin{attn}\label{attn:hom-width}
It follows from~\ref{attn:odd-reduced} that $\tohw_R(L)=\ohw_R(L)-1$. The same
holds true in the case of the even Khovanov homology over $\Q$:
$\thw_\Q(L)=\hw_\Q(L)-1$ (see~\cite{Kh-patterns}).
\end{attn}

\begin{defin}\label{def:thick-thin}
A link $L$ is said to be {\em homologically thin} over a ring $R$, or simply
$R$H-thin, if $\hw_R(L)=2$. $L$ is {\em homologically thick}, or $R$H-thick,
otherwise. We define {\em odd-homologically} thin and thick, or simply
$R$OH-thin and $R$OH-thick, links similarly. 
\end{defin}

\begin{thm}[Lee, Ozsv\'ath--Rasmussen--Szab\'o,
Manolescu--Ozsv\'ath~\cite{Lee,Khovanov-odd,QA-links}]\label{thm:alt-thin}
Quasi-alternating links (see Section~\ref{sec:quasi-alt} for the definition)
are $R$H-thin and $R$OH-thin for every ring $R$. In particular, this is
true for non-split alternating links.
\end{thm}

\begin{thm}[Khovanov~\cite{Kh-patterns}]\label{thm:adequate-thick}
Adequate links are $R$H-thick for every $R$.
\end{thm}

\begin{attn}
Homological thickness of a link $L$ often does not depend on the base ring.
The first prime knot with $\hw_{\Q}(L)<\hw_{\Z_2}(L)$ and
$\hw_{\Q}(L)<\hw_{\Z}(L)$ is $15^n_{41127}$ with 15 crossings (see
Figure~\ref{fig:odd-thick}). The first prime
knot that is $\Q$H-thin but $\Z$H-thick, $16^n_{197566}$, has 16 crossings
(see Figure~\ref{fig:diff-thickness}). Its mirror image,
$\overline{16}^n_{197566}$ is both $\Q$H- and $\Z$H-thin but is $\Z_2$H-thick
with $\CalH^{-8,-21}(\overline{16}^n_{197566};\Z_2)\simeq\Z_2$, for example,
because of the Universal Coefficient Theorem. Also observe that
$\CalH^{9,25}(16^n_{197566};\Z)$ and
$\CalH^{-8,-25}(\overline{16}^n_{197566};\Z)$ have $4$-torsion, shown in a
small box in the tables.
\end{attn}

\begin{rem}
\label{rem:knots-enum}
Throughout this paper we use the following notation for knots: knots with 10
crossings or less are numbered according to the Rolfsen's knot
table~\cite{Rolfsen-book} and knots with 11 crossings or more are numbered
according to the knot table from Knotscape~\cite{Knotscape}. Mirror images of
knots from either table are denoted with a bar on top. For example,
$\overline{9}_{46}$ is the mirror image of the knot number 46 with 9 crossings
from the Rolfsen's table and $16^n_{197566}$ is the non-alternating
knot number 197566 with 16 crossings from the Knotscape's one.
\end{rem}

\begin{figure}
\centerline{\input knot-15n_41127-table-Red}
\par\medskip
$\hw_{\Q}=3,\qquad \thw_{\Q}=2,\qquad \hw_{\Z}=4,\qquad \thw_{\Z}=3$
\par\bigskip\bigskip
\centerline{\input knot-15n_41127-table-Odd}
\par\medskip
$\ohw_{\Q}=4,\qquad \tohw_{\Q}=3,\qquad \ohw_{\Z}=4,\qquad \tohw_{\Z}=3$

\caption{Integral reduced even Khovanov homology (above) and odd Khovanov
Homology (below) of the knot $15^n_{41127}$}\label{fig:odd-thick}
\end{figure}

\begin{figure}
\centerline{\resizebox{\hsize}{!}{\input knot-16n_197566-table}}
\par\medskip
The free part of $\CalH(16^n_{197566};\Z)$ is supported on diagonals 
$j-2i=7$ and $j-2i=9$. On the other hand, there is $2$-torsion on the diagonal
$j-2i=5$. Therefore, $16^n_{197566}$ is $\Q$H-thin, but $\Z$H-thick and
$\Z_2$H-thick.
\par\bigskip\bigskip\bigskip
\centerline{\resizebox{\hsize}{!}{\input knot-16n_-197566-table}}
\par\medskip
$\CalH(\overline{16}^n_{197566};\Z)$ is supported on diagonals 
$j-2i=-7$ and $j-2i=-9$.\\
But there is $2$-torsion on the diagonal $j-2i=-7$.\\
Therefore, $\overline{16}^n_{197566}$ is $\Q$H-thin and
$\Z$H-thin, but $\Z_2$H-thick.
\par\bigskip
\caption{Integral Khovanov homology of the knots $16^n_{197566}$ and
$\overline{16}^n_{197566}$}\label{fig:diff-thickness}
\end{figure}

\begin{attn}
Odd Khovanov homology is often thicker over $\Z$ than the even one. This
is crucial for applications (see Section~\ref{sec:Kh-appl}).
On the other hand, $\tohw_{\Q}(L)\le\thw_{\Q}(L)$ for all but one prime knot
with at most $15$ crossings. The homology for this knot, $15^n_{41127}$, is
shown in Figure~\ref{fig:odd-thick}. Please observe that
$\tCalH_{\odd}(15^n_{41127})$ has $3$-torsion (in gradings $(-2,-2)$ and
$(-1,0)$), while $\tCalH(15^n_{41127})$ has none.
\end{attn}

\subsection{Lee spectral sequence and the Knight-Move Conjecture}
In~\cite{Lee} Eun Soo Lee introduced a structure of a spectral sequence
on the rational 
Khovanov chain complex $\CalC(D;\Q)$ of a link diagram $D$. Namely, Lee
defined a differential $d':\CalC(D;\Q)\to\CalC(D;\Q)$ of bidegree $(1,4)$ by
setting
\begin{equation}\label{eq:Lee-diff}
\begin{split}
m':A\otimes A\to A:&\quad
m'(1{\otimes}1)=m'(1{\otimes}X)=m'(X{\otimes}1)=0,\quad m'(X{\otimes}X)=1\\
\GD':A\to A\otimes A:&\quad \GD'(1)=0,\qquad\GD'(X)=1\otimes 1\\
\end{split}
\end{equation}
It is straightforward to verify that $d'$ is indeed a differential and that it
anti-commutes with $d$, that is $d\circ d'+d'\circ d=0$. This makes
$(\CalC(D;\Q),d,d')$ into a double complex. Let $d'_*$ be the differential
induced by $d'$ on $\CalH(L;\Q)$. Lee proved that $d'_*$ is functorial, that
is, it commutes with isomorphisms induced on $\CalH(L;\Q)$ by isotopies of
$L$. It follows that there exists a spectral sequence with
$(E_1,d_1)=(\CalC(D;\Q),d)$ and
$(E_2,d_2)=(\CalH(L;\Q),d'_*)$ that converges to the homology of the total
(filtered) complex of $\CalC(D;\Q)$ with respect to the differential $d+d'$.
It is called the {\em Lee spectral sequence}. The differentials $d_n$ in this
spectral sequence have bidegree $(1,4(n-1))$.

\begin{thm}[Lee~\cite{Lee}]\label{thm:Lee}
If $L$ is an oriented link with $\#L$ components, then
$H(\mathrm{Total}(\CalC(D;\Q)),d+d')$,
the limit of the Lee spectral sequence, consists
of $2^{n-1}$ copies of $\Q\oplus\Q$, each located in a specific homological
grading that is explicitly defined by linking numbers of the components of
$L$. In particular, if $L$ is a knot, then the Lee spectral sequence converges
to $\Q\oplus\Q$ localed in homological grading~$0$.
\end{thm}

The following Theorem is the cornerstone in the definition of the Rasmussen
invariant, one of the main applications of the Khovanov homology (see
Section~\ref{sec:Rasmussen}).

\begin{thm}[Rasmussen~\cite{Jake-Milnor}]\label{thm:Jake-infinity}
If $L$ is a knot, then the two copies of $\Q$ in the limiting term of the Lee
spectral sequence for $L$ are ``neighbors'', that is, their $q$-gradings are
different by $2$.
\end{thm}

\begin{cor}\label{cor:lee-collapse}
If the Lee spectral sequence for a link $L$ collapses after the second page
(that is, $d_n=0$
for $n\ge3$), then $\CalH(L;\Q)$ consists of one ``pawn-move'' pair in
homological grading~$0$ and multiple ``knight-move'' pairs, shown below, with
appropriate grading shifts.
\[
\hbox{\rm Pawn-move pair: }
\begin{array}{|c|}\oldhline\Q\\ \oldhline\Q\\ \oldhline\end{array}\qquad\qquad
\hbox{\rm Knight-move pair: }
\begin{array}{|c|c|}\oldhline&\Q\\ \oldhline&\\ \oldhline \Q&\\ \oldhline\end{array}
\]
\end{cor}

\begin{cor}
Since $d_3$ has bidegree $(1,8)$, the Lee spectral sequence collapses after
the second page for all knots with homological width $2$ or $3$, in
particular, for all alternating and quasi-alternating knots. 
Hence, Corollary~\ref{cor:lee-collapse} can be applied to such knots.
\end{cor}

\begin{knightmove}[Garoufalidis--Khovanov--Bar-Natan~\cite{BN-first,Kh-Jones}]
\label{knight-move}
The conclusion of Corollary~\ref{cor:lee-collapse} is true for every knot.
\end{knightmove}

\begin{rem}
There are currently no known counter-examples to the Knight-Move Conjecture.
In fact, the Lee spectral sequence can be proved (in one way or another) to
collapse after the second page for every known example of Khovanov homology.
\end{rem}

\begin{rem}
While the Lee spectral sequence exists over any ring $R$, the statement of
Theorem~\ref{thm:Lee} does not hold true for all of them. In particular, it is
wrong over~$\Z_2$. In this case, though, a similar theory was constructed by
Paul Turner~\cite{Turner-Char2}. In fact, his construction works for reduced
Khovanov homology as well because of~\ref{thm:me-Z2}. While
Theorems~\ref{thm:Lee} and~\ref{thm:Jake-infinity} are still true over $\Z_p$
with odd prime $p$, Knight-Move Conjecture is known to be false over such
rings~\cite{BN-fast}.
\end{rem}

\begin{rem}
The Lee spectral sequence has no analog in the odd and reduced Khovanov
homology theories (except over $\Z_2$, as noted above, where the two theories
coincide). The Knight-Move Conjecture has no analog in these theories either.
\end{rem}

\def\Cpic#1{%
  \CalC\!\left(\vcenter{\hbox{\includegraphics[scale=0.35]{#1}}}\right)\!}
\def\Hpic#1{%
  \CalH\!\left(\vcenter{\hbox{\includegraphics[scale=0.35]{#1}}}\right)\!}
\subsection{Long exact sequence of the Khovanov homology}
One of the most useful tools in studying Khovanov homology is the long exact
sequence that categorifies the Kauffman's {\em unoriented} skein relation for
the Jones polynomial~\cite{Kh-Jones}. If we forget about the grading, then it
is clear from the construction from Section~\ref{sec:Kh-complex} that
$\,\Cpic{hsmooth_Xing-black}\,$ is a subcomplex of
$\,\Cpic{just_Xing-black}\,$ and $\,\Cpic{vsmooth_Xing-black}\simeq
\Cpic{just_Xing-black}/\Cpic{hsmooth_Xing-black}\,$ (see also
Figure~\ref{fig:Kh-Hopf}). Here, 
$\vcenter{\hbox{\includegraphics[scale=0.35]{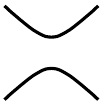}}}$ and
$\vcenter{\hbox{\includegraphics[scale=0.35]{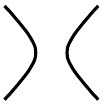}}}$ depict
link diagrams where a single crossing
$\vcenter{\hbox{\includegraphics[scale=0.35]{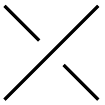}}}$ is
resolved in a negative or, respectively, positive direction. This results in a
short exact sequence of non-graded chain complexes:
\begin{equation}\label{eq:short-exact}
0\lra\Cpic{hsmooth_Xing-black}\stackrel{in}{\lra}\Cpic{just_Xing-black}
\stackrel{p}{\lra}\Cpic{vsmooth_Xing-black}\lra0,
\end{equation}
where $in$ is the inclusion and $p$ is the projection.

In order to introduce grading into~\eqref{eq:short-exact}, we need to consider
the cases when the crossing to be resolved is either positive or negative.
We get (see~\cite{Jake-comparison}):
\begin{equation}\label{eq:short-exact-orient}\begin{split}
0\lra\Cpic{hsmooth_Xing-black}\{2{+}3\Go\}[1{+}\Go]\stackrel{in}{\lra}
\Cpic{pos_Xing-black}\stackrel{p}{\lra}
\Cpic{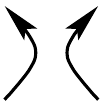}\{1\}\lra0,\\[\smallskipamount]
0\lra\Cpic{smooth_Xing-nomark}\{-1\}\stackrel{in}{\lra}
\Cpic{neg_Xing-black}\stackrel{p}{\lra}
\Cpic{hsmooth_Xing-black}\{1{+}3\Go\}[\Go]\lra0,
\end{split}\end{equation}
where $\Go$ is the difference between the numbers of negative crossings in the
unoriented resolution
$\vcenter{\hbox{\includegraphics[scale=0.35]{hsmooth_Xing-black}}}$ (it has to
be oriented somehow in order to define its Khovanov chain complex) and 
in the original diagram. The notation $\CalC[k]$ is used to represent a
shift in the homological grading of a complex $\CalC$ by $k$. The graded
versions of $in$ and $p$ are both homogeneous, that is, have bidegree $(0,0)$.

By passing to homology in~\eqref{eq:short-exact-orient}, we get the following
result.

\begin{thm}[Khovanov, Viro,
Rasmussen~\cite{Kh-Jones,Viro-defs,Jake-comparison}]\label{thm:Kh-exact}
The Khovanov homology is subject to the following long exact sequences:
\begin{equation}\label{eq:Kh-exact}\begin{split}
\cdots\lra\Hpic{smooth_Xing-nomark}\{1\}\stackrel{\partial}{\lra}
\Hpic{hsmooth_Xing-black}\{2{+}3\Go\}[1{+}\Go]\stackrel{in_*}{\lra}
\Hpic{pos_Xing-black}\stackrel{p_*}{\lra}
\Hpic{smooth_Xing-nomark}\{1\}\lra\cdots\\[\smallskipamount]
\cdots\lra\Hpic{smooth_Xing-nomark}\{-1\}\stackrel{in_*}{\lra}
\Hpic{neg_Xing-black}\stackrel{p_*}{\lra}
\Hpic{hsmooth_Xing-black}\{1{+}3\Go\}[\Go]\stackrel{\partial}{\lra}
\Hpic{smooth_Xing-nomark}\{-1\}\lra\cdots
\end{split}\end{equation}
where $in_*$ and $p_*$ are homogeneous and $\partial$ is the connecting
differential and has bidegree $(1,0)$.
\end{thm}

\begin{rem}
Long exact sequences~\eqref{eq:Kh-exact} work equally well over any ring $R$
and for every version of the Khovanov homology, including the odd one
(see~\cite{Khovanov-odd}). This is both a blessing and a curse. On one hand,
this means that all of the properties of the even Khovanov homology that are
proved using these long exact sequences (and most of them are) hold
automatically true for the odd Khovanov homology as well. 
On the other hand, this makes it very hard to find explanations to many
differences between these homology theories.
\end{rem}

\section{Applications of the Khovanov homology}\label{sec:Kh-appl}

In this section we collect some of the more prominent applications of the
Khovanov homology theories. This list is by no means complete and is chosen to
provide the reader with a broader view on the type of problems that can be
solved with a help of the Khovanov homology. We make a special effort to
compare the performance of different versions of the homology, where
applicable.

\subsection{Rasmussen invariant and bounds on the slice genus}
\label{sec:Rasmussen}
One of the most important applications of the Khovanov's construction so far
was obtained by Jacob Rasmussen in~2004. In~\cite{Jake-Milnor} he used the
structure of the Lee spectral sequence to define a new invariant of knots that
gives a lower bound on the slice genus. More specifically, for a knot $L$, its
Rasmussen invariant $s(L)$ is defined as the mean $q$-grading of the two copies
of $\Q$ that remain in the homological grading $0$ of the limiting term of the
Lee spectral sequence, see Theorem~\ref{thm:Jake-infinity}. Since the
$q$-gradings of these $\Q$'s are odd and are different by $2$, the Rasmussen
invariant is an even integer.

\begin{thm}[Rasmussen~\cite{Jake-Milnor}]\label{thm:Jake-s}
Let $L$ be a knot and $s(L)$ its Rasmussen invariant. Then
\begin{itemize}
\item $|s(L)|\le 2g_s(L)$, where $g_s(L)$ is the slice genus of $L$, that is, 
the smallest possible genus of a smoothly embedded surface in the $4$-ball
$D^4$ that has $L\subset S^3=\partial D^4$ as its boundary;
\item $s(L)=\sigma(L)$ for alternating $L$, where $\sigma(L)$ is the
signature of $L$;
\item $s(L)=2g_s(L)=2g(L)$ for a knot $L$ that possesses a planar diagram with
positive crossings only, where $g(L)$ is the genus of $L$;
\item if knots $L_-$ and $L_+$ have diagrams that are different at a single
crossing in such a way that this crossing is negative in $L_-$ and positive in
$L_+$, then $s(L_-)\le s(L_+)\le s(L_-)+2$.
\end{itemize}
\end{thm}

\begin{cor}
$s(T_{p,q})=(p-1)(q-1)$ for $p,q>0$, where $T_{p,q}$ is the $(p,q)$-torus
knot. This implies the Milnor Conjecture, first proved by Kronheimer and Mrowka
in~1993 using the gauge theory~\cite{Kronheimer-Mrowka}. This conjecture
states that the slice genus (and, hence, the genus) of $T_{p,q}$ equals
$\frac12((p-1)(q-1))$. The upper bound on the slice genus is straightforward,
so the lower bound provided by the Rasmussen invariant is sharp.
\end{cor}

\begin{rem}
Although the Rasmussen invariant was originally defined for knots only, its
definition was later extended to the case of links by Anna Beliakova and
Stephan Wehrli~\cite{Beliakova-colored}.
\end{rem}

The Rasmussen invariant can be used to 
search for knots that are topologically locally-flatly slice but are not
smoothly slice (see~\cite{Sh-Rasmussen}). A knot is slice if its slice genus
is~$0$. Theorem~\ref{thm:Jake-s} implies that knots with non-trivial Rasmussen
invariant
are not smoothly slice. On the other hand, it was proved by
Freedman~\cite{Freedman} that knots with Alexander polynomial~1 are
topologically locally-flatly slice. There are 82 knots with up to 16 crossings
that possess these two properties~\cite{Sh-Rasmussen}. Each such knot gives
rise to a family of exotic $\R^4$~\cite[Exercise~9.4.23]{Gompf}. It is worth
noticing that most of these $82$ examples were not previously known.

The Rasmussen invariant was also used~\cite{Plamenevskaya,Sh-Rasmussen} to
deduce the combinatorial proof of the Slice-Bennequin Inequality. This
inequality states that
\begin{equation}\label{eq:SB-ineq}
g_s(\widehat{\beta})\le\frac12(w(\beta)-k+1),
\end{equation}
where $\beta$ is a braid on $k$ strands with the closure $\widehat{\beta}$ and
$w(\beta)$ is its writhe number.
The Slice-Bennequin Inequality provides one of the upper bounds for the
Thurston-Bennequin number of Legendrian links (see below). It was originally
proved by Lee Rudolph~\cite{Rudolph-quasipositivity} using the gauge theory.
The approach via the Rasmussen invariant and Khovanov homology avoids
gauge theory and symplectic Floer theory and results in a purely combinatorial
proof.

\begin{rem}
Since the Rasmussen's construction relies on the existence and convergence of
the Lee spectral sequence, the Rasmussen invariant can only be defined for the
even non-reduced Khovanov homology. In fact, a knot might not have any
rational homology in the homological grading $0$ of the odd Khovanov homology
at all, see Figure~\ref{fig:12n_475-hom}.
\end{rem}

\subsection{Bounds on the Thurston-Bennequin number}\label{sec:TB}
Another useful applications of the Khovanov homology is in finding
upper bounds on the Thurston-Bennequin number of Legendrian links.
Consider $\R^3$ equipped with the standard contact structure $dz-y\,dx$.
A link $K\subset\R^3$ is said to be {\em Legendrian} if it is everywhere
tangent to
the $2$-dimensional plane distribution defined as the kernel of this $1$-form.
Given a Legendrian link $K$, one defines its {\em Thurston--Bennequin number},
$tb(K)$, as the linking number of $K$ with its push-off $K'$ obtained using
a vector field that is tangent to the contact planes but orthogonal to the
tangent vector field of $K$. Roughly speaking, $tb(K)$ measures the framing of
the contact plane field around $K$. It is well-known that the TB-number can be
made arbitrarily small within the same class of topological links via
stabilization, but is bounded from above.

\begin{defin}\label{def:TB-bound}
For a given {\em topological} link $L$, let $\tbb(L)$, the {\em TB-bound}
of $L$, be the maximal possible TB-number among all the Legendrian
representatives of $L$. In other words, $\tbb(L)=\max_K\{tb(K)\}$, where
$K$ runs over all the Legendrian links in $\R^3$ that are topologically
isotopic to $L$.
\end{defin}

Finding TB-bounds for links attracts considerable interest lately, since they
can be used to demonstrate that certain contact structures on $\R^3$ are not
isomorphic to the standard one. Such bounds  can be obtained from the Bennequin
and Slice-Bennequin inequalities, degrees of HOMFLY-PT and Kauffman
polynomials, Knot Floer homology, and so on (see~\cite{Ng-TB-bound}
for more details). The TB-bound coming from the Kauffman polynomial is usually
one of the strongest, since most of the others incorporate another invariant
of Legendrian links, the rotation number, into the inequality.
In~\cite{Ng-TB-bound}, Lenhard Ng used Khovanov homology to define a new
bound on the TB-number.

\begin{thm}[Ng~\cite{Ng-TB-bound}]\label{thm:Ng-TB}
Let $L$ be an oriented link. Then
\begin{equation}\label{eq:Ng-TB}
\tbb(L)\le\min\Bigl\{k\big|\bigoplus_{j-i=k}\CalH^{i,j}(L;R)\not=0\Bigr\}.
\end{equation}
Moreover, this bound is sharp for alternating links.
\end{thm}

This Khovanov bound on the TB-number is often better than those that were
known before. There are only two prime knots with up to 13 crossings for which
the Khovanov bound is worse than the one coming from the Kauffman
polynomial~\cite{Ng-TB-bound}. There are 45 such knots with at most 15
crossings.

\begin{example}
Figure~\ref{fig:TB-bounds} shows computations of the Khovanov TB-bound for the
$(4,-5)$-torus knot. The Khovanov homology groups in~\eqref{eq:Ng-TB} can be
used over any ring $R$, and this example shows that the bound coming from the
integral homology is sometimes better than the one from the rational one, due
to a strategically placed torsion. It is interesting to note that the integral
Khovanov bound of $-20$ is computed incorrectly in~\cite{Ng-TB-bound}. In
particular, this was one of the cases where Ng thought that the Kauffman
polynomial provides a better one. In fact, the TB-bound of $-20$ is sharp for
this knot.
\end{example}

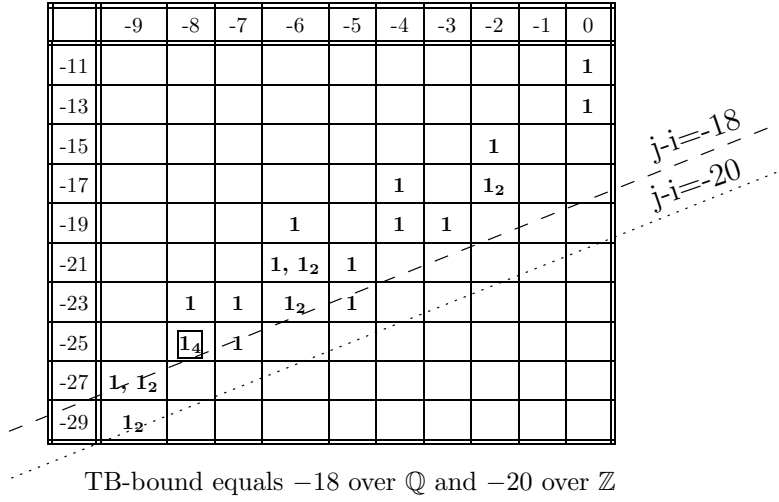
\begin{figure}
\centerline{%
\vbox{\halign{\hfill#\hfill\cr\input{torus-4_-5-table-diag.pspdftex}\cr
\noalign{\vskip -2\bigskipamount}
TB-bound equals $-18$ over $\Q$ and $-20$ over $\Z$\qquad\qquad\cr}}}
\caption{Khovanov TB-bound for the $(4,-5)$-torus knot}\label{fig:TB-bounds}
\end{figure}

The proof of Theorem~\ref{thm:Ng-TB} is based on the long exact
sequences~\eqref{eq:Kh-exact} and, hence, can be applied verbatim to the
reduced as well as odd Khovanov homology. By making appropriate adjustments to
the grading, we immediately get~\cite{Sh-OddKhovanov} that
\begin{align}
\tbb(L)\le&-1+\min\Bigl\{k\big|\bigoplus_{j-i=k}\tCalH^{i,j}(L;R)\not=0\Bigr\}\\
\tbb(L)\le&-1+\min\Bigl\{k\big|
\bigoplus_{j-i=k}\tCalH_{\odd}^{i,j}(L;R)\not=0\Bigr\}.
\end{align}

As it turns out, the odd Khovanov TB-bound is often better than the even one.
In fact, computations performed in~\cite{Sh-OddKhovanov} show that
the odd Khovanov homology provide the best upper bound on the TB-number
among all currently known ones for all prime knots with at most 15 crossings.
In particular, the odd Khovanov TB-bound equals the Kauffman one on all the
$45$ knots with at most $15$ crossings where the latter is better than the
even Khovanov TB-bound.

\begin{example}
Odd Khovanov TB-bound is better than the even one and equals to the Kauffman
one for the knot $12^n_{475}$, as shown in Figure~\ref{fig:12n_475-hom}.
\end{example}

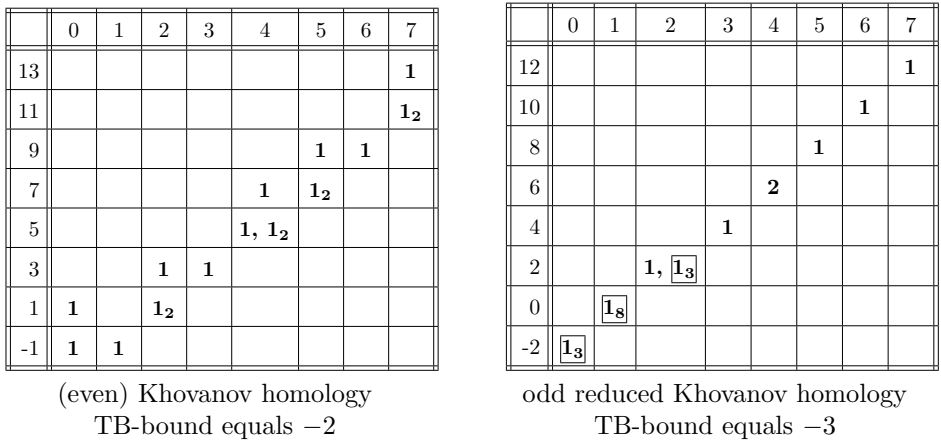
\begin{figure}
\centerline{%
\vbox{\halign{\hfill#\hfill\cr
\resizebox{0.47\hsize}{!}{\input{knot-12n_475-table}}\cr
(even) Khovanov homology\cr TB-bound equals $-2$\cr}}\qquad
\vbox{\halign{\hfill#\hfill\cr
\resizebox{0.47\hsize}{!}{\input{knot-12n_475-table-Odd}}\cr
odd reduced Khovanov homology\cr TB-bound equals $-3$\cr}}}
\caption{Khovanov TB-bounds for the knot $12^n_{475}$}\label{fig:12n_475-hom}
\end{figure}

\subsection{Finding quasi-alternating knots}\label{sec:quasi-alt}
Quasi-alternating links were introduced by Ozsv\'ath and Szab\'o
in~\cite{OS-spectral} as a way to generalize the class of alternating links.

\begin{defin}\label{def:QA-links} The class $\CalQ$ of quasi-alternating links
is the smallest set of links such that
\begin{itemize}
\item the unknot belongs to $\CalQ$;
\item if a link $L$ has a planar diagram $D$ such that the two
resolutions of this diagram at one crossing represent two links,
$L_0$ and $L_1$, with the properties that $L_0,L_1\in\CalQ$ and 
$\det(L)=\det(L_0)+\det(L_1)$, then $L\in\CalQ$ as well.
\end{itemize}
\end{defin}

\begin{rem}
It is well-known that all non-split alternating links are quasi-alternating.
\end{rem}

The main motivation for studying quasi-alternating links is the fact that
the double branched covers of $S^3$ along such links are so-called $L$-spaces.
A $3$-manifold $M$ is called an {\em $L$-space} if the order of its first
homology group $H_1$ is finite and equals the rank of the Heegaard--Floer
homology of $M$ (see~\cite{OS-spectral}). Unfortunately, due to the recursive
style of Definition~\ref{def:QA-links}, it is often highly non-trivial to
prove that a given link is quasi-alternating. It is equally challenging to
show that it is not.

To determine that a link is not quasi-alternating, one usually employs the fact
that such links have homologically thin Khovanov homology over $\Z$ and Knot
Floer homology over $\Z_2$ (see~\cite{QA-links}).
Thus, $\Z$H-thick knots are not quasi-alternating. There are 12 such knots
with up to 10 crossings. Most of the others can be shown to be
quasi-alternating by various constructions. After the work of Champanerkar and
Kofman~\cite{Kofman-Co-QA}, there were only two knots left, $9_{46}$ and
$10_{140}$, for which it was not known whether they are quasi-alternating or
not. Both of them have homologically thin Khovanov and Knot Floer homology.

As it turns out, odd Khovanov homology is much better at detecting
quasi-alternating knots. The proof of the fact that such knots are $\Z$H-thin
is based  on the long exact sequences~\eqref{eq:Kh-exact} and, therefore, can
be applied verbatim to the odd homology as well~\cite{Khovanov-odd}.
Computations show~\cite{Sh-OddKhovanov} that the knots $9_{46}$ and $10_{140}$
have
homologically thick odd Khovanov homology and, hence, are not
quasi-alternating, see Figures~\ref{fig:9_46-hom} and~\ref{fig:10_140-hom}.

\begin{figure}
\centerline{%
\vbox{\halign{\hfill#\hfill\cr
\resizebox{0.47\hsize}{!}{\input{pretzel-3-3-n3-table-Red}}\cr
(even) reduced Khovanov homology\cr}}\qquad
\vbox{\halign{\hfill#\hfill\cr
\resizebox{0.47\hsize}{!}{\input{pretzel-3-3-n3-table-Odd}}\cr
odd reduced Khovanov homology\cr}}}
\caption{Khovanov homology of $9_{46}$, the $(3,3,-3)$-pretzel
knot}\label{fig:9_46-hom}
\end{figure}

\begin{figure}
\centerline{%
\vbox{\halign{\hfill#\hfill\cr
\resizebox{0.47\hsize}{!}{\input{pretzel-3-4-n3-table-Red}}\cr
(even) reduced Khovanov homology\cr}}\qquad
\vbox{\halign{\hfill#\hfill\cr
\resizebox{0.47\hsize}{!}{\input{pretzel-3-4-n3-table-Odd}}\cr
odd reduced Khovanov homology\cr}}}
\caption{Khovanov homology of $10_{140}$, the $(3,4,-3)$-pretzel
knot}\label{fig:10_140-hom}
\end{figure}

\begin{rem}\label{rem:odd-torsion}
It is worth mentioning that the knots $9_{46}$ and $10_{140}$ are $(3,3,-3)$-
and $(3,4,-3)$-pretzel knots, respectively (see Figure~\ref{fig:pretzels} for
the definition). Computations show that $(n,n,-n)$- and $(n,n+1,-n)$-pretzel
links for $n\le 6$ all have torsion of order $n$ outside of the main diagonal
that supports the free part of the homology. This suggest a certain
$n$-fold symmetry on the odd Khovanov chain complexes for these
pretzel links that cannot be explained by the construction.
\end{rem}

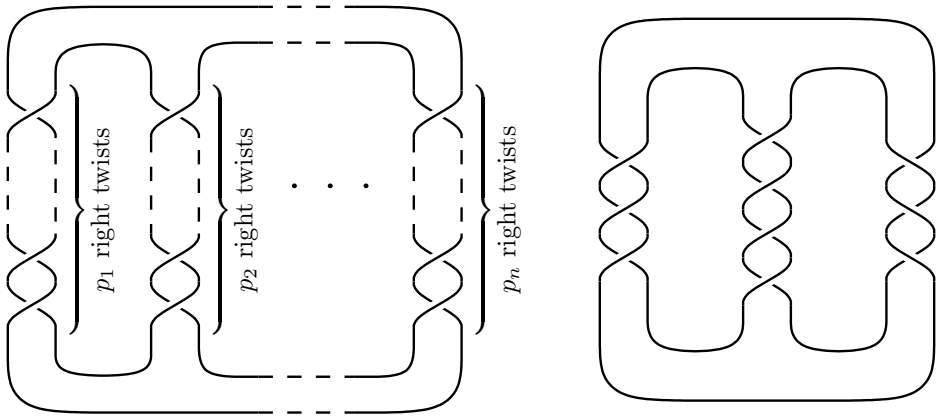
\begin{figure}
\centerline{$\vcenter{\hbox{\input{pretzel-link.pspdftex}}}\qquad\qquad
\vcenter{\hbox{\input{pretzel-3_4_n3.pspdftex}}}$}
\caption{$(p_1,p_2,\dots,p_n)$-pretzel link and $(3,4,-3)$-pretzel knot
$10_{140}$}\label{fig:pretzels}
\end{figure}

\begin{rem}
Joshua Greene has recently determined~\cite{Greene-QA} all quasi-alternating
pretzel links by considering $4$-manifolds that are bounded by the 
branched double covers of the links.
In particular, he found several knots that are not
quasi-alternating, yet both H-thin and OH-thin. The smallest such knot
is $11^n_{50}$. 
\end{rem}

\subsection{Detection of the unknot}
It was recently showed by Kronheimer and Mrowka~\cite{Kronheimer-Mrowka-unknot}
that Khovanov homology detects the unknot. More specifically, they proved that
a knot is the unknot if and only if its reduced Khovanov homology has rank $1$.
This development is a major step towards proving a long-standing conjecture
that the Jones polynomial itself detects the unknot.

\end{document}

%% file: KhoHo-tables.tex
\let\TSp\thinspace
\def\DSp{\thinspace\thinspace}
\def\lbr{\raise 1pt\hbox{[}}
\def\rbr{\raise 1pt\hbox{]}}

\def\gobble#1{}
\def\hline{\multispan\numcolumns\hrulefill\cr}
\def\torframe#1{\vtop{\vbox{\hrule\hbox{\vrule\strut #1\vrule}}\hrule}}
\def\dblhline{\hline height 0.16667em\gobble&\emptyline\cr\hline}

\newbox\tablebox

%% file: trefoil-Kh.tex
\def\emptyline{&&&&&}
\def\numcolumns{7}%

\setbox\tablebox\vbox{\offinterlineskip\ialign{%
\vrule\TSp\vrule #\strut&\DSp\hfil #\DSp\vrule\TSp\vrule&
\DSp\hfil\bf #\hfil\DSp\vrule&
\DSp\hfil\bf #\hfil\DSp\vrule&
\DSp\hfil\bf #\hfil\DSp\vrule&
\DSp\hfil\bf #\hfil\DSp\vrule&#\TSp\vrule\cr
\dblhline
height 11pt depth 4pt&&
\rm\DSp0\DSp&
\rm\DSp1\DSp&
\rm\DSp2\DSp&
\rm\DSp3\DSp&\cr\dblhline
height 13pt depth 5pt&9&
&
&
&
\DSp1\DSp&
\cr\hline
height 13pt depth 5pt&7&
&
&
&
$\mathbf{1_2}$&
\cr\hline
height 13pt depth 5pt&5&
&
&
\DSp1\DSp&
&
\cr\hline
height 13pt depth 5pt&3&
\DSp1\DSp&
&
&
&
\cr\hline
height 13pt depth 5pt&1&
\DSp1\DSp&
&
&
&
\cr
\dblhline
}}

\box\tablebox

%% file: Xing_signs.pspdftex
\begin{picture}(0,0)%
\includegraphics{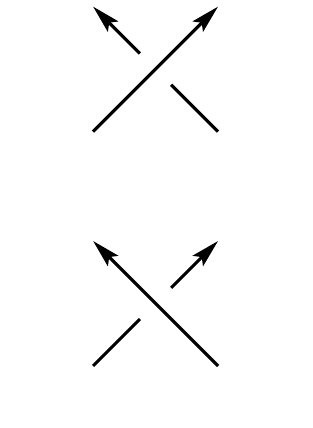}%
\end{picture}%
\setlength{\unitlength}{3947sp}%
\begingroup\makeatletter\ifx\SetFigFont\undefined%
\gdef\SetFigFont#1#2#3#4#5{%
  \reset@font\fontsize{#1}{#2pt}%
  \fontfamily{#3}\fontseries{#4}\fontshape{#5}%
  \selectfont}%
\fi\endgroup%
\begin{picture}(1491,2041)(530,-1780)
\put(1276,-586){\makebox(0,0)[b]{\smash{{\SetFigFont{12}{14.4}{\rmdefault}{\mddefault}{\updefault}{\color[rgb]{0,0,0}positive crossing}%
}}}}
\put(1276,-1711){\makebox(0,0)[b]{\smash{{\SetFigFont{12}{14.4}{\rmdefault}{\mddefault}{\updefault}{\color[rgb]{0,0,0}negative crossing}%
}}}}
\end{picture}%

%% file: markers.pspdftex
\begin{picture}(0,0)%
\includegraphics{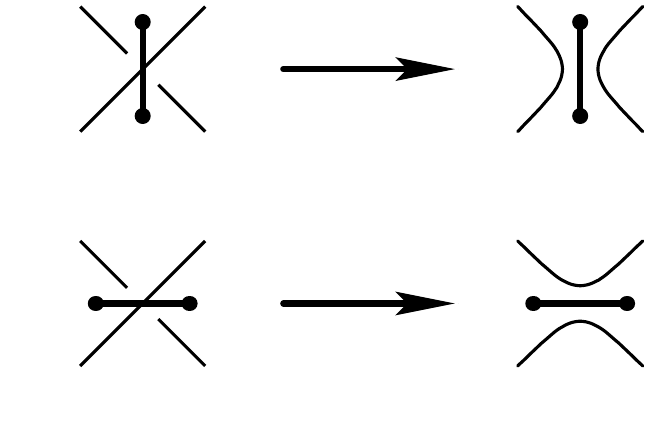}%
\end{picture}%
\setlength{\unitlength}{3947sp}%
\begingroup\makeatletter\ifx\SetFigFont\undefined%
\gdef\SetFigFont#1#2#3#4#5{%
  \reset@font\fontsize{#1}{#2pt}%
  \fontfamily{#3}\fontseries{#4}\fontshape{#5}%
  \selectfont}%
\fi\endgroup%
\begin{picture}(3107,2041)(516,-1780)
\put(938,-98){\makebox(0,0)[lb]{\smash{{\SetFigFont{12}{14.4}{\rmdefault}{\mddefault}{\updefault}{\color[rgb]{0,0,0}$+$}%
}}}}
\put(1201,-586){\makebox(0,0)[b]{\smash{{\SetFigFont{12}{14.4}{\rmdefault}{\mddefault}{\updefault}{\color[rgb]{0,0,0}positive marker}%
}}}}
\put(1216,-1029){\makebox(0,0)[b]{\smash{{\SetFigFont{12}{14.4}{\rmdefault}{\mddefault}{\updefault}{\color[rgb]{0,0,0}$-$}%
}}}}
\put(1201,-1711){\makebox(0,0)[b]{\smash{{\SetFigFont{12}{14.4}{\rmdefault}{\mddefault}{\updefault}{\color[rgb]{0,0,0}negative marker}%
}}}}
\end{picture}%

%% file: differential.pspdftex
\begin{picture}(0,0)%
\includegraphics{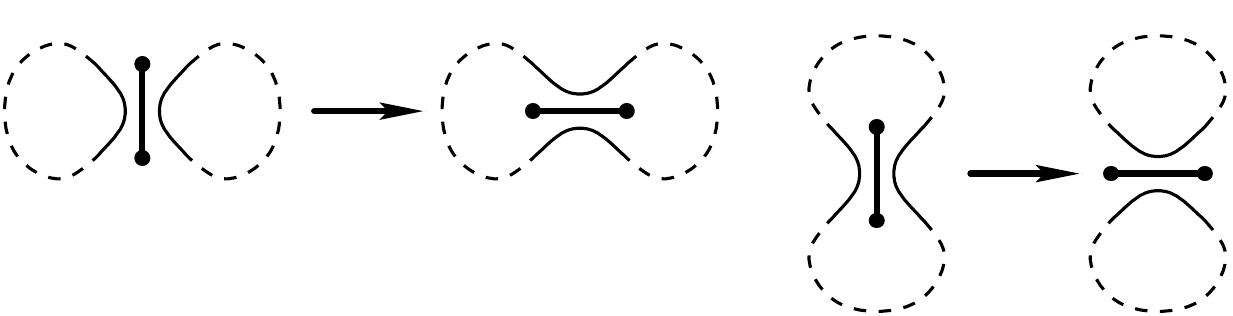}%
\end{picture}%
\setlength{\unitlength}{3947sp}%
\begingroup\makeatletter\ifx\SetFigFont\undefined%
\gdef\SetFigFont#1#2#3#4#5{%
  \reset@font\fontsize{#1}{#2pt}%
  \fontfamily{#3}\fontseries{#4}\fontshape{#5}%
  \selectfont}%
\fi\endgroup%
\begin{picture}(5905,1505)(593,-894)
\put(4846, 67){\makebox(0,0)[lb]{\smash{{\SetFigFont{12}{14.4}{\rmdefault}{\mddefault}{\updefault}{\color[rgb]{0,0,0}$+$}%
}}}}
\put(4576,464){\makebox(0,0)[rb]{\smash{{\SetFigFont{12}{14.4}{\rmdefault}{\mddefault}{\updefault}{\color[rgb]{0,0,0}$A$}%
}}}}
\put(6413,-136){\makebox(0,0)[lb]{\smash{{\SetFigFont{12}{14.4}{\rmdefault}{\mddefault}{\updefault}{\color[rgb]{0,0,0}$-$}%
}}}}
\put(5926,464){\makebox(0,0)[rb]{\smash{{\SetFigFont{12}{14.4}{\rmdefault}{\mddefault}{\updefault}{\color[rgb]{0,0,0}$A\otimes A$}%
}}}}
\put(5461,-136){\makebox(0,0)[b]{\smash{{\SetFigFont{12}{14.4}{\rmdefault}{\mddefault}{\updefault}{\color[rgb]{0,0,0}$\Delta$}%
}}}}
\put(1283,330){\makebox(0,0)[lb]{\smash{{\SetFigFont{12}{14.4}{\rmdefault}{\mddefault}{\updefault}{\color[rgb]{0,0,0}$+$}%
}}}}
\put(1276,464){\makebox(0,0)[b]{\smash{{\SetFigFont{12}{14.4}{\rmdefault}{\mddefault}{\updefault}{\color[rgb]{0,0,0}$A\otimes A$}%
}}}}
\put(3638,164){\makebox(0,0)[lb]{\smash{{\SetFigFont{12}{14.4}{\rmdefault}{\mddefault}{\updefault}{\color[rgb]{0,0,0}$-$}%
}}}}
\put(3376,389){\makebox(0,0)[b]{\smash{{\SetFigFont{12}{14.4}{\rmdefault}{\mddefault}{\updefault}{\color[rgb]{0,0,0}$A$}%
}}}}
\put(2318,164){\makebox(0,0)[b]{\smash{{\SetFigFont{12}{14.4}{\rmdefault}{\mddefault}{\updefault}{\color[rgb]{0,0,0}$m$}%
}}}}
\end{picture}%

%% file: HopfLink-complex.pspdftex
\begin{picture}(0,0)%
\includegraphics{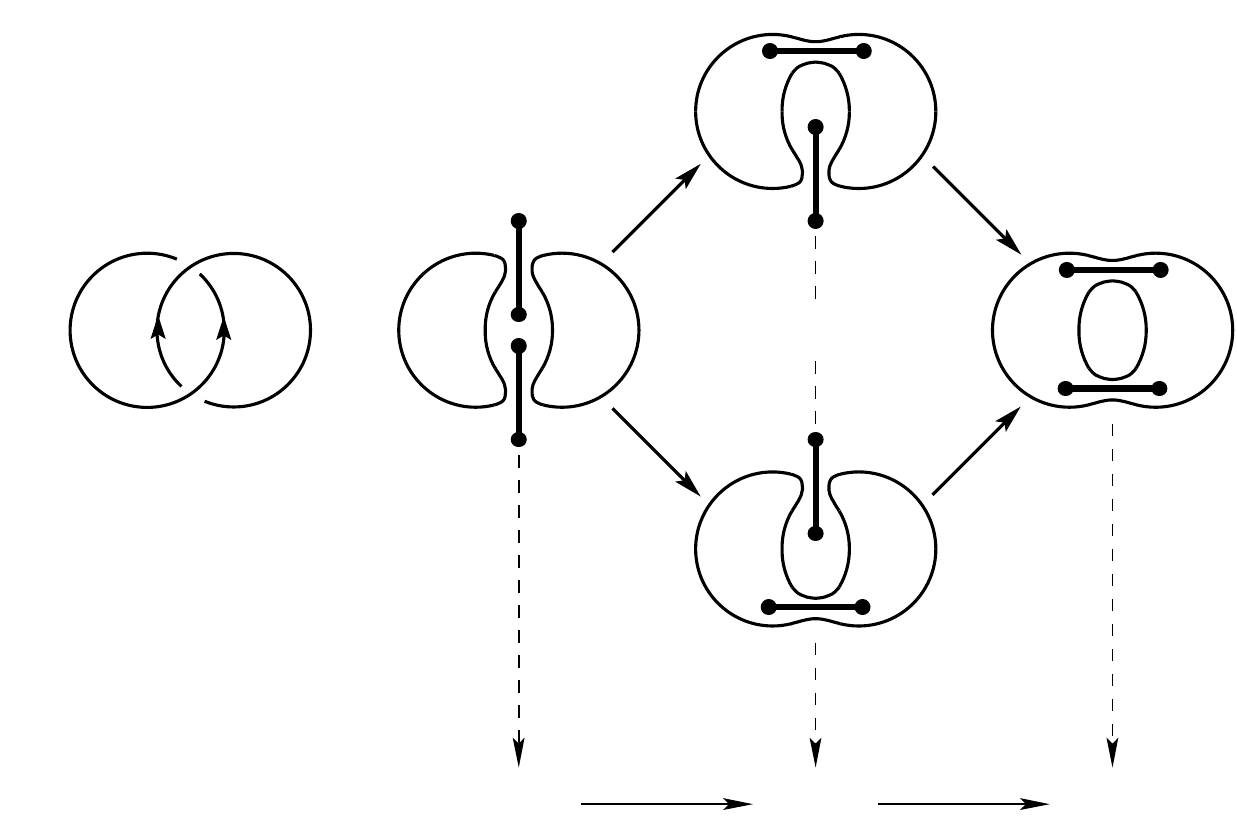}%
\end{picture}%
\setlength{\unitlength}{3947sp}%
\begingroup\makeatletter\ifx\SetFigFont\undefined%
\gdef\SetFigFont#1#2#3#4#5{%
  \reset@font\fontsize{#1}{#2pt}%
  \fontfamily{#3}\fontseries{#4}\fontshape{#5}%
  \selectfont}%
\fi\endgroup%
\begin{picture}(5932,3970)(586,-3959)
\put(4769,-379){\makebox(0,0)[lb]{\smash{{\SetFigFont{12}{14.4}{\rmdefault}{\mddefault}{\updefault}{\color[rgb]{0,0,0}$-$}%
}}}}
\put(4546,-1172){\makebox(0,0)[lb]{\smash{{\SetFigFont{12}{14.4}{\rmdefault}{\mddefault}{\updefault}{\color[rgb]{0,0,0}$+$}%
}}}}
\put(3121,-983){\makebox(0,0)[lb]{\smash{{\SetFigFont{12}{14.4}{\rmdefault}{\mddefault}{\updefault}{\color[rgb]{0,0,0}$+$}%
}}}}
\put(3121,-2222){\makebox(0,0)[lb]{\smash{{\SetFigFont{12}{14.4}{\rmdefault}{\mddefault}{\updefault}{\color[rgb]{0,0,0}$+$}%
}}}}
\put(6194,-1429){\makebox(0,0)[lb]{\smash{{\SetFigFont{12}{14.4}{\rmdefault}{\mddefault}{\updefault}{\color[rgb]{0,0,0}$-$}%
}}}}
\put(6188,-1766){\makebox(0,0)[lb]{\smash{{\SetFigFont{12}{14.4}{\rmdefault}{\mddefault}{\updefault}{\color[rgb]{0,0,0}$-$}%
}}}}
\put(3685,-921){\rotatebox{45.0}{\makebox(0,0)[b]{\smash{{\SetFigFont{12}{14.4}{\rmdefault}{\mddefault}{\updefault}{\color[rgb]{0,0,0}$m$}%
}}}}}
\put(3791,-2095){\rotatebox{315.0}{\makebox(0,0)[b]{\smash{{\SetFigFont{12}{14.4}{\rmdefault}{\mddefault}{\updefault}{\color[rgb]{0,0,0}$m$}%
}}}}}
\put(5221,-2086){\rotatebox{45.0}{\makebox(0,0)[b]{\smash{{\SetFigFont{12}{14.4}{\rmdefault}{\mddefault}{\updefault}{\color[rgb]{0,0,0}$-\GD$}%
}}}}}
\put(3901,-2836){\makebox(0,0)[rb]{\smash{{\SetFigFont{10}{12.0}{\rmdefault}{\mddefault}{\updefault}{\color[rgb]{0,0,0}$\Gs(s_{+-}){=}0$}%
}}}}
\put(3901,-3061){\makebox(0,0)[rb]{\smash{{\SetFigFont{10}{12.0}{\rmdefault}{\mddefault}{\updefault}{\color[rgb]{0,0,0}$i(s_{+-}){=}1$}%
}}}}
\put(3901,-3286){\makebox(0,0)[rb]{\smash{{\SetFigFont{10}{12.0}{\rmdefault}{\mddefault}{\updefault}{\color[rgb]{0,0,0}$j(s_{+-}){=}3$}%
}}}}
\put(4763,-2816){\makebox(0,0)[lb]{\smash{{\SetFigFont{12}{14.4}{\rmdefault}{\mddefault}{\updefault}{\color[rgb]{0,0,0}$-$}%
}}}}
\put(4546,-2033){\makebox(0,0)[lb]{\smash{{\SetFigFont{12}{14.4}{\rmdefault}{\mddefault}{\updefault}{\color[rgb]{0,0,0}$+$}%
}}}}
\put(5776,-361){\makebox(0,0)[lb]{\smash{{\SetFigFont{10}{12.0}{\rmdefault}{\mddefault}{\updefault}{\color[rgb]{0,0,0}$\Gs(s_{--}){=}{-}2$}%
}}}}
\put(5776,-586){\makebox(0,0)[lb]{\smash{{\SetFigFont{10}{12.0}{\rmdefault}{\mddefault}{\updefault}{\color[rgb]{0,0,0}$i(s_{--}){=}2$}%
}}}}
\put(5776,-811){\makebox(0,0)[lb]{\smash{{\SetFigFont{10}{12.0}{\rmdefault}{\mddefault}{\updefault}{\color[rgb]{0,0,0}$j(s_{--}){=}4$}%
}}}}
\put(3901,-136){\makebox(0,0)[rb]{\smash{{\SetFigFont{10}{12.0}{\rmdefault}{\mddefault}{\updefault}{\color[rgb]{0,0,0}$\Gs(s_{-+}){=}0$}%
}}}}
\put(3901,-361){\makebox(0,0)[rb]{\smash{{\SetFigFont{10}{12.0}{\rmdefault}{\mddefault}{\updefault}{\color[rgb]{0,0,0}$i(s_{-+}){=}1$}%
}}}}
\put(3901,-586){\makebox(0,0)[rb]{\smash{{\SetFigFont{10}{12.0}{\rmdefault}{\mddefault}{\updefault}{\color[rgb]{0,0,0}$j(s_{-+}){=}3$}%
}}}}
\put(3751,-3736){\makebox(0,0)[b]{\smash{{\SetFigFont{12}{14.4}{\rmdefault}{\mddefault}{\updefault}{\color[rgb]{0,0,0}$d^0$}%
}}}}
\put(5176,-3736){\makebox(0,0)[b]{\smash{{\SetFigFont{12}{14.4}{\rmdefault}{\mddefault}{\updefault}{\color[rgb]{0,0,0}$d^1$}%
}}}}
\put(2926,-2161){\makebox(0,0)[rb]{\smash{{\SetFigFont{10}{12.0}{\rmdefault}{\mddefault}{\updefault}{\color[rgb]{0,0,0}$\Gs(s_{++}){=}2$}%
}}}}
\put(2926,-2386){\makebox(0,0)[rb]{\smash{{\SetFigFont{10}{12.0}{\rmdefault}{\mddefault}{\updefault}{\color[rgb]{0,0,0}$i(s_{++}){=}0$}%
}}}}
\put(2926,-2611){\makebox(0,0)[rb]{\smash{{\SetFigFont{10}{12.0}{\rmdefault}{\mddefault}{\updefault}{\color[rgb]{0,0,0}$j(s_{++}){=}2$}%
}}}}
\put(601,-1611){\makebox(0,0)[lb]{\smash{{\SetFigFont{12}{14.4}{\rmdefault}{\mddefault}{\updefault}{\color[rgb]{0,0,0}$\CalC\left(\vrule height 0.3truein depth 0.3truein width 0pt\hskip 0.96truein\right)\!=$}%
}}}}
\put(5330,-934){\rotatebox{315.0}{\makebox(0,0)[b]{\smash{{\SetFigFont{12}{14.4}{\rmdefault}{\mddefault}{\updefault}{\color[rgb]{0,0,0}$\GD$}%
}}}}}
\put(4498,-1612){\makebox(0,0)[b]{\smash{{\SetFigFont{12}{14.4}{\rmdefault}{\mddefault}{\updefault}{\color[rgb]{0,0,0}$\displaystyle\bigoplus$}%
}}}}
\put(1501,-1786){\makebox(0,0)[b]{\smash{{\SetFigFont{12}{14.4}{\rmdefault}{\mddefault}{\updefault}{\color[rgb]{0,0,0}$2$}%
}}}}
\put(1501,-1186){\makebox(0,0)[b]{\smash{{\SetFigFont{12}{14.4}{\rmdefault}{\mddefault}{\updefault}{\color[rgb]{0,0,0}$1$}%
}}}}
\put(5026,-211){\makebox(0,0)[lb]{\smash{{\SetFigFont{12}{14.4}{\rmdefault}{\mddefault}{\updefault}{\color[rgb]{0,0,0}$A\{3\}$}%
}}}}
\put(5026,-2986){\makebox(0,0)[lb]{\smash{{\SetFigFont{12}{14.4}{\rmdefault}{\mddefault}{\updefault}{\color[rgb]{0,0,0}$A\{3\}$}%
}}}}
\put(5776,-1111){\makebox(0,0)[lb]{\smash{{\SetFigFont{12}{14.4}{\rmdefault}{\mddefault}{\updefault}{\color[rgb]{0,0,0}$(A{\otimes}A)\{4\}$}%
}}}}
\put(1051,-2311){\makebox(0,0)[lb]{\smash{{\SetFigFont{10}{12.0}{\rmdefault}{\mddefault}{\updefault}{\color[rgb]{0,0,0}$w(D){=}2$}%
}}}}
\put(2926,-1111){\makebox(0,0)[rb]{\smash{{\SetFigFont{12}{14.4}{\rmdefault}{\mddefault}{\updefault}{\color[rgb]{0,0,0}$A{\otimes}A\{2\}$}%
}}}}
\put(4501,-3886){\makebox(0,0)[b]{\smash{{\SetFigFont{12}{14.4}{\rmdefault}{\mddefault}{\updefault}{\color[rgb]{0,0,0}$\CalC^1(D)$}%
}}}}
\put(5926,-3886){\makebox(0,0)[b]{\smash{{\SetFigFont{12}{14.4}{\rmdefault}{\mddefault}{\updefault}{\color[rgb]{0,0,0}$\CalC^2(D)$}%
}}}}
\put(3076,-3886){\makebox(0,0)[b]{\smash{{\SetFigFont{12}{14.4}{\rmdefault}{\mddefault}{\updefault}{\color[rgb]{0,0,0}$\CalC^0(D)$}%
}}}}
\put(2626,-3886){\makebox(0,0)[rb]{\smash{{\SetFigFont{12}{14.4}{\rmdefault}{\mddefault}{\updefault}{\color[rgb]{0,0,0}$\CalC(D)=\Bigl(\hskip 2.8truein\Bigr)$\hskip -3.05truein}%
}}}}
\end{picture}%

%% file: odd-arrows.pspdftex
\begin{picture}(0,0)%
\includegraphics{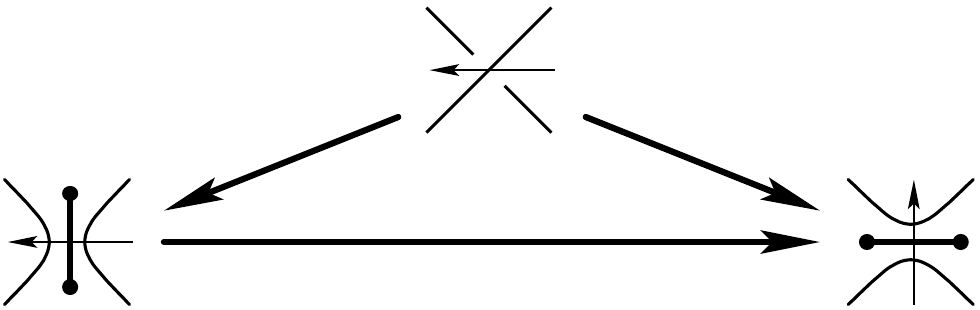}%
\end{picture}%
\setlength{\unitlength}{3947sp}%
\begingroup\makeatletter\ifx\SetFigFont\undefined%
\gdef\SetFigFont#1#2#3#4#5{%
  \reset@font\fontsize{#1}{#2pt}%
  \fontfamily{#3}\fontseries{#4}\fontshape{#5}%
  \selectfont}%
\fi\endgroup%
\begin{picture}(4694,1469)(639,-1208)
\put(5288,-811){\makebox(0,0)[lb]{\smash{{\SetFigFont{12}{14.4}{\rmdefault}{\mddefault}{\updefault}{\color[rgb]{0,0,0}$-$}%
}}}}
\put(1021,-600){\makebox(0,0)[lb]{\smash{{\SetFigFont{12}{14.4}{\rmdefault}{\mddefault}{\updefault}{\color[rgb]{0,0,0}$+$}%
}}}}
\put(3001,-811){\makebox(0,0)[b]{\smash{{\SetFigFont{10}{12.0}{\rmdefault}{\mddefault}{\updefault}{\color[rgb]{0,0,0}rotate the arrow by $90^\circ$ clockwise}%
}}}}
\end{picture}%

%% file: differential-odd.pspdftex
\begin{picture}(0,0)%
\includegraphics{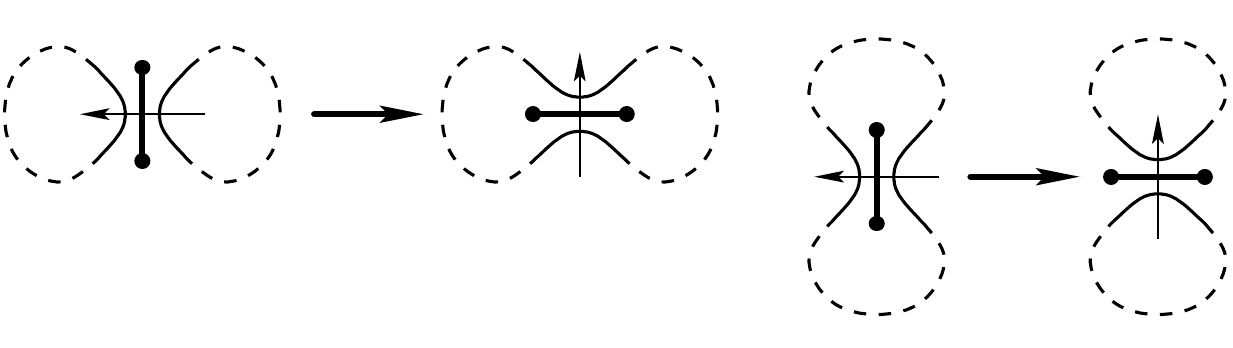}%
\end{picture}%
\setlength{\unitlength}{3947sp}%
\begingroup\makeatletter\ifx\SetFigFont\undefined%
\gdef\SetFigFont#1#2#3#4#5{%
  \reset@font\fontsize{#1}{#2pt}%
  \fontfamily{#3}\fontseries{#4}\fontshape{#5}%
  \selectfont}%
\fi\endgroup%
\begin{picture}(5905,1653)(593,-1030)
\put(1283,330){\makebox(0,0)[lb]{\smash{{\SetFigFont{12}{14.4}{\rmdefault}{\mddefault}{\updefault}{\color[rgb]{0,0,0}$+$}%
}}}}
\put(751,464){\makebox(0,0)[rb]{\smash{{\SetFigFont{12}{14.4}{\rmdefault}{\mddefault}{\updefault}{\color[rgb]{0,0,0}$X^+_1$}%
}}}}
\put(1801,464){\makebox(0,0)[lb]{\smash{{\SetFigFont{12}{14.4}{\rmdefault}{\mddefault}{\updefault}{\color[rgb]{0,0,0}$X^+_2$}%
}}}}
\put(3638,164){\makebox(0,0)[lb]{\smash{{\SetFigFont{12}{14.4}{\rmdefault}{\mddefault}{\updefault}{\color[rgb]{0,0,0}$-$}%
}}}}
\put(3676,464){\makebox(0,0)[rb]{\smash{{\SetFigFont{12}{14.4}{\rmdefault}{\mddefault}{\updefault}{\color[rgb]{0,0,0}$X^-_1$}%
}}}}
\put(2318,189){\makebox(0,0)[b]{\smash{{\SetFigFont{12}{14.4}{\rmdefault}{\mddefault}{\updefault}{\color[rgb]{0,0,0}$m_{\odd}$}%
}}}}
\put(4846, 67){\makebox(0,0)[lb]{\smash{{\SetFigFont{12}{14.4}{\rmdefault}{\mddefault}{\updefault}{\color[rgb]{0,0,0}$+$}%
}}}}
\put(6413,-136){\makebox(0,0)[lb]{\smash{{\SetFigFont{12}{14.4}{\rmdefault}{\mddefault}{\updefault}{\color[rgb]{0,0,0}$-$}%
}}}}
\put(5926,464){\makebox(0,0)[rb]{\smash{{\SetFigFont{12}{14.4}{\rmdefault}{\mddefault}{\updefault}{\color[rgb]{0,0,0}$X^-_2$}%
}}}}
\put(5926,-961){\makebox(0,0)[rb]{\smash{{\SetFigFont{12}{14.4}{\rmdefault}{\mddefault}{\updefault}{\color[rgb]{0,0,0}$X^-_1$}%
}}}}
\put(5461,-111){\makebox(0,0)[b]{\smash{{\SetFigFont{12}{14.4}{\rmdefault}{\mddefault}{\updefault}{\color[rgb]{0,0,0}$\Delta_{\odd}$}%
}}}}
\put(4576,464){\makebox(0,0)[rb]{\smash{{\SetFigFont{12}{14.4}{\rmdefault}{\mddefault}{\updefault}{\color[rgb]{0,0,0}$X^+_1$}%
}}}}
\end{picture}%

%% file: HopfLink-complex-odd.pspdftex
\begin{picture}(0,0)%
\includegraphics{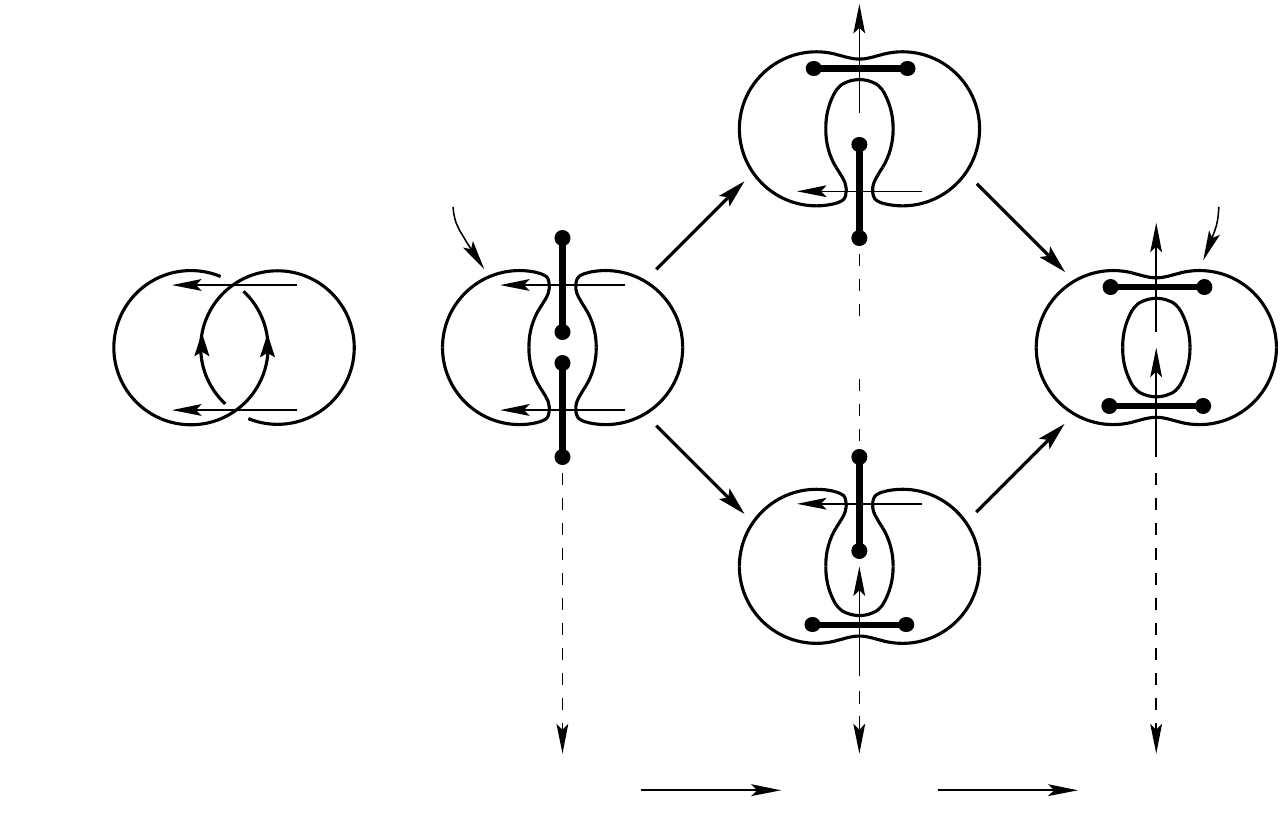}%
\end{picture}%
\setlength{\unitlength}{3947sp}%
\begingroup\makeatletter\ifx\SetFigFont\undefined%
\gdef\SetFigFont#1#2#3#4#5{%
  \reset@font\fontsize{#1}{#2pt}%
  \fontfamily{#3}\fontseries{#4}\fontshape{#5}%
  \selectfont}%
\fi\endgroup%
\begin{picture}(6142,3910)(376,-3809)
\put(4769,-379){\makebox(0,0)[lb]{\smash{{\SetFigFont{12}{14.4}{\rmdefault}{\mddefault}{\updefault}{\color[rgb]{0,0,0}$-$}%
}}}}
\put(4546,-1172){\makebox(0,0)[lb]{\smash{{\SetFigFont{12}{14.4}{\rmdefault}{\mddefault}{\updefault}{\color[rgb]{0,0,0}$+$}%
}}}}
\put(6194,-1429){\makebox(0,0)[lb]{\smash{{\SetFigFont{12}{14.4}{\rmdefault}{\mddefault}{\updefault}{\color[rgb]{0,0,0}$-$}%
}}}}
\put(6188,-1766){\makebox(0,0)[lb]{\smash{{\SetFigFont{12}{14.4}{\rmdefault}{\mddefault}{\updefault}{\color[rgb]{0,0,0}$-$}%
}}}}
\put(3685,-921){\rotatebox{45.0}{\makebox(0,0)[b]{\smash{{\SetFigFont{8}{9.6}{\rmdefault}{\mddefault}{\updefault}{\color[rgb]{0,0,0}$/(X_1{-}X_2)$}%
}}}}}
\put(3791,-2095){\rotatebox{315.0}{\makebox(0,0)[b]{\smash{{\SetFigFont{8}{9.6}{\rmdefault}{\mddefault}{\updefault}{\color[rgb]{0,0,0}$/(X_1{-}X_2)$}%
}}}}}
\put(5221,-2086){\rotatebox{45.0}{\makebox(0,0)[b]{\smash{{\SetFigFont{8}{9.6}{\rmdefault}{\mddefault}{\updefault}{\color[rgb]{0,0,0}$(X_2{-}X_1)\wedge$}%
}}}}}
\put(4763,-2816){\makebox(0,0)[lb]{\smash{{\SetFigFont{12}{14.4}{\rmdefault}{\mddefault}{\updefault}{\color[rgb]{0,0,0}$-$}%
}}}}
\put(4546,-2033){\makebox(0,0)[lb]{\smash{{\SetFigFont{12}{14.4}{\rmdefault}{\mddefault}{\updefault}{\color[rgb]{0,0,0}$+$}%
}}}}
\put(5176,-3586){\makebox(0,0)[b]{\smash{{\SetFigFont{12}{14.4}{\rmdefault}{\mddefault}{\updefault}{\color[rgb]{0,0,0}$d^1_{\odd}$}%
}}}}
\put(3751,-3586){\makebox(0,0)[b]{\smash{{\SetFigFont{12}{14.4}{\rmdefault}{\mddefault}{\updefault}{\color[rgb]{0,0,0}$d^0_{\odd}$}%
}}}}
\put(3121,-983){\makebox(0,0)[lb]{\smash{{\SetFigFont{12}{14.4}{\rmdefault}{\mddefault}{\updefault}{\color[rgb]{0,0,0}$+$}%
}}}}
\put(3121,-2222){\makebox(0,0)[lb]{\smash{{\SetFigFont{12}{14.4}{\rmdefault}{\mddefault}{\updefault}{\color[rgb]{0,0,0}$+$}%
}}}}
\put(5330,-934){\rotatebox{315.0}{\makebox(0,0)[b]{\smash{{\SetFigFont{8}{9.6}{\rmdefault}{\mddefault}{\updefault}{\color[rgb]{0,0,0}$(X_1{-}X_2)\wedge$}%
}}}}}
\put(4498,-1612){\makebox(0,0)[b]{\smash{{\SetFigFont{12}{14.4}{\rmdefault}{\mddefault}{\updefault}{\color[rgb]{0,0,0}$\displaystyle\bigoplus$}%
}}}}
\put(5026,-211){\makebox(0,0)[lb]{\smash{{\SetFigFont{10}{12.0}{\rmdefault}{\mddefault}{\updefault}{\color[rgb]{0,0,0}$\Lambda\!^*V(X_1)$}%
}}}}
\put(5026,-2986){\makebox(0,0)[lb]{\smash{{\SetFigFont{10}{12.0}{\rmdefault}{\mddefault}{\updefault}{\color[rgb]{0,0,0}$\Lambda\!^*V(X_1)$}%
}}}}
\put(2581,-3736){\makebox(0,0)[rb]{\smash{{\SetFigFont{12}{14.4}{\rmdefault}{\mddefault}{\updefault}{\color[rgb]{0,0,0}$\CalC_{\odd}(D)=\Bigl(\hskip 2.9truein\Bigr)$\hskip -3.15truein}%
}}}}
\put(391,-1611){\makebox(0,0)[lb]{\smash{{\SetFigFont{12}{14.4}{\rmdefault}{\mddefault}{\updefault}{\color[rgb]{0,0,0}$\CalC_{\odd}\left(\vrule height 0.3truein depth 0.3truein width 0pt\hskip 0.96truein\right)\!=$}%
}}}}
\put(3076,-3736){\makebox(0,0)[b]{\smash{{\SetFigFont{12}{14.4}{\rmdefault}{\mddefault}{\updefault}{\color[rgb]{0,0,0}$\CalC_{\odd}^0(D)$}%
}}}}
\put(4501,-3736){\makebox(0,0)[b]{\smash{{\SetFigFont{12}{14.4}{\rmdefault}{\mddefault}{\updefault}{\color[rgb]{0,0,0}$\CalC_{\odd}^1(D)$}%
}}}}
\put(5926,-3736){\makebox(0,0)[b]{\smash{{\SetFigFont{12}{14.4}{\rmdefault}{\mddefault}{\updefault}{\color[rgb]{0,0,0}$\CalC_{\odd}^2(D)$}%
}}}}
\put(2551,-1486){\makebox(0,0)[lb]{\smash{{\SetFigFont{12}{14.4}{\rmdefault}{\mddefault}{\updefault}{\color[rgb]{0,0,0}$X_1$}%
}}}}
\put(3976,-436){\makebox(0,0)[lb]{\smash{{\SetFigFont{12}{14.4}{\rmdefault}{\mddefault}{\updefault}{\color[rgb]{0,0,0}$X_1$}%
}}}}
\put(3976,-2536){\makebox(0,0)[lb]{\smash{{\SetFigFont{12}{14.4}{\rmdefault}{\mddefault}{\updefault}{\color[rgb]{0,0,0}$X_1$}%
}}}}
\put(3226,-1486){\makebox(0,0)[lb]{\smash{{\SetFigFont{12}{14.4}{\rmdefault}{\mddefault}{\updefault}{\color[rgb]{0,0,0}$X_2$}%
}}}}
\put(5401,-1486){\makebox(0,0)[lb]{\smash{{\SetFigFont{12}{14.4}{\rmdefault}{\mddefault}{\updefault}{\color[rgb]{0,0,0}$X_1$}%
}}}}
\put(5776,-1561){\makebox(0,0)[lb]{\smash{{\SetFigFont{12}{14.4}{\rmdefault}{\mddefault}{\updefault}{\color[rgb]{0,0,0}$X_2$}%
}}}}
\put(3001,-811){\makebox(0,0)[rb]{\smash{{\SetFigFont{10}{12.0}{\rmdefault}{\mddefault}{\updefault}{\color[rgb]{0,0,0}$\Lambda\!^*V(X_1,X_2)$}%
}}}}
\put(5701,-811){\makebox(0,0)[lb]{\smash{{\SetFigFont{10}{12.0}{\rmdefault}{\mddefault}{\updefault}{\color[rgb]{0,0,0}$\Lambda\!^*V(X_1,X_2)$}%
}}}}
\end{picture}%

%% file: knot-15n_41127-table-Red.tex
\def\emptyline{&&&&&&&&&&&&}
\def\numcolumns{14}%

\setbox\tablebox\vbox{\offinterlineskip\ialign{%
\vrule\TSp\vrule #\strut&\DSp\hfil #\DSp\vrule\TSp\vrule&
\DSp\hfil\bf #\hfil\DSp\vrule&
\DSp\hfil\bf #\hfil\DSp\vrule&
\DSp\hfil\bf #\hfil\DSp\vrule&
\DSp\hfil\bf #\hfil\DSp\vrule&
\DSp\hfil\bf #\hfil\DSp\vrule&
\DSp\hfil\bf #\hfil\DSp\vrule&
\DSp\hfil\bf #\hfil\DSp\vrule&
\DSp\hfil\bf #\hfil\DSp\vrule&
\DSp\hfil\bf #\hfil\DSp\vrule&
\DSp\hfil\bf #\hfil\DSp\vrule&
\DSp\hfil\bf #\hfil\DSp\vrule&#\TSp\vrule\cr
\dblhline
height 11pt depth 4pt&&
\rm\DSp-7\DSp&
\rm\DSp-6\DSp&
\rm\DSp-5\DSp&
\rm\DSp-4\DSp&
\rm\DSp-3\DSp&
\rm\DSp-2\DSp&
\rm\DSp-1\DSp&
\rm\DSp0\DSp&
\rm\DSp1\DSp&
\rm\DSp2\DSp&
\rm\DSp3\DSp&\cr\dblhline
height 13pt depth 5pt&8&
&
&
&
&
&
&
&
&
&
&
\DSp1\DSp&
\cr\hline
height 13pt depth 5pt&6&
&
&
&
&
&
&
&
&
&
\DSp1\DSp&
&
\cr\hline
height 13pt depth 5pt&4&
&
&
&
&
&
&
&
&
\DSp1\DSp&
&
&
\cr\hline
height 13pt depth 5pt&2&
&
&
&
&
&
&
\DSp1\DSp&
\DSp2\DSp&
&
&
&
\cr\hline
height 13pt depth 5pt&0&
&
&
&
&
&
\DSp1\DSp&
&
$\mathbf{1_2}$&
&
&
&
\cr\hline
height 13pt depth 5pt&-2&
&
&
&
&
\DSp1\DSp&
\DSp1\DSp&
&
&
&
&
&
\cr\hline
height 13pt depth 5pt&-4&
&
&
&
\DSp2\DSp&
\DSp1\DSp&
&
&
&
&
&
&
\cr\hline
height 13pt depth 5pt&-6&
&
&
\DSp1\DSp&
&
&
&
&
&
&
&
&
\cr\hline
height 13pt depth 5pt&-8&
&
\DSp1\DSp&
&
&
&
&
&
&
&
&
&
\cr\hline
height 13pt depth 5pt&-10&
\DSp1\DSp&
&
&
&
&
&
&
&
&
&
&
\cr
\dblhline
}}

\box\tablebox

%% file: knot-15n_41127-table-Odd.tex
\def\emptyline{&&&&&&&&&&&&}
\def\numcolumns{14}%

\setbox\tablebox\vbox{\offinterlineskip\ialign{%
\vrule\TSp\vrule #\strut&\DSp\hfil #\DSp\vrule\TSp\vrule&
\DSp\hfil\bf #\hfil\DSp\vrule&
\DSp\hfil\bf #\hfil\DSp\vrule&
\DSp\hfil\bf #\hfil\DSp\vrule&
\DSp\hfil\bf #\hfil\DSp\vrule&
\DSp\hfil\bf #\hfil\DSp\vrule&
\DSp\hfil\bf #\hfil\DSp\vrule&
\DSp\hfil\bf #\hfil\DSp\vrule&
\DSp\hfil\bf #\hfil\DSp\vrule&
\DSp\hfil\bf #\hfil\DSp\vrule&
\DSp\hfil\bf #\hfil\DSp\vrule&
\DSp\hfil\bf #\hfil\DSp\vrule&#\TSp\vrule\cr
\dblhline
height 11pt depth 4pt&&
\rm\DSp-7\DSp&
\rm\DSp-6\DSp&
\rm\DSp-5\DSp&
\rm\DSp-4\DSp&
\rm\DSp-3\DSp&
\rm\DSp-2\DSp&
\rm\DSp-1\DSp&
\rm\DSp0\DSp&
\rm\DSp1\DSp&
\rm\DSp2\DSp&
\rm\DSp3\DSp&\cr\dblhline
height 13pt depth 5pt&8&
&
&
&
&
&
&
&
&
&
&
\DSp1\DSp&
\cr\hline
height 13pt depth 5pt&6&
&
&
&
&
&
&
&
&
&
\DSp1\DSp&
&
\cr\hline
height 13pt depth 5pt&4&
&
&
&
&
&
&
&
&
\DSp1\DSp&
&
&
\cr\hline
height 13pt depth 5pt&2&
&
&
&
&
&
&
&
1, $\mathbf{1_2}$&
&
&
&
\cr\hline
height 13pt depth 5pt&0&
&
&
&
&
&
&
$\torframe{$\mathbf{1_6}$}$&
\DSp1\DSp&
&
&
&
\cr\hline
height 13pt depth 5pt&-2&
&
&
&
&
&
$\torframe{$\mathbf{1_6}$}$&
&
&
&
&
&
\cr\hline
height 13pt depth 5pt&-4&
&
&
&
\DSp1\DSp&
$\mathbf{1_2}$&
&
&
&
&
&
&
\cr\hline
height 13pt depth 5pt&-6&
&
&
\DSp1\DSp&
&
&
&
&
&
&
&
&
\cr\hline
height 13pt depth 5pt&-8&
&
\DSp1\DSp&
&
&
&
&
&
&
&
&
&
\cr\hline
height 13pt depth 5pt&-10&
\DSp1\DSp&
&
&
&
&
&
&
&
&
&
&
\cr
\dblhline
}}

\box\tablebox

%% file: knot-16n_197566-table.tex
\def\emptyline{&&&&&&&&&&&&&&}
\def\numcolumns{16}%

\setbox\tablebox\vbox{\offinterlineskip\ialign{%
\vrule\TSp\vrule #\strut&\DSp\hfil #\DSp\vrule\TSp\vrule&
\DSp\hfil\bf #\hfil\DSp\vrule&
\DSp\hfil\bf #\hfil\DSp\vrule&
\DSp\hfil\bf #\hfil\DSp\vrule&
\DSp\hfil\bf #\hfil\DSp\vrule&
\DSp\hfil\bf #\hfil\DSp\vrule&
\DSp\hfil\bf #\hfil\DSp\vrule&
\DSp\hfil\bf #\hfil\DSp\vrule&
\DSp\hfil\bf #\hfil\DSp\vrule&
\DSp\hfil\bf #\hfil\DSp\vrule&
\DSp\hfil\bf #\hfil\DSp\vrule&
\DSp\hfil\bf #\hfil\DSp\vrule&
\DSp\hfil\bf #\hfil\DSp\vrule&
\DSp\hfil\bf #\hfil\DSp\vrule&#\TSp\vrule\cr
\dblhline
height 11pt depth 4pt&&
\rm\DSp-2\DSp&
\rm\DSp-1\DSp&
\rm\DSp0\DSp&
\rm\DSp1\DSp&
\rm\DSp2\DSp&
\rm\DSp3\DSp&
\rm\DSp4\DSp&
\rm\DSp5\DSp&
\rm\DSp6\DSp&
\rm\DSp7\DSp&
\rm\DSp8\DSp&
\rm\DSp9\DSp&
\rm\DSp10\DSp&\cr\dblhline
height 13pt depth 5pt&29&
&
&
&
&
&
&
&
&
&
&
&
&
\DSp1\DSp&
\cr\hline
height 13pt depth 5pt&27&
&
&
&
&
&
&
&
&
&
&
&
\DSp4\DSp&
$\mathbf{1_2}$&
\cr\hline
height 13pt depth 5pt&25&
&
&
&
&
&
&
&
&
&
&
\DSp7\DSp&
1, $\mathbf{3_2}\torframe{$\mathbf{1_4}$}$&
&
\cr\hline
height 13pt depth 5pt&23&
&
&
&
&
&
&
&
&
&
\DSp12\DSp&
4, $\mathbf{8_2}$&
&
&
\cr\hline
height 13pt depth 5pt&21&
&
&
&
&
&
&
&
&
\DSp15\DSp&
7, $\mathbf{12_2}$&
$\mathbf{1_2}$&
&
&
\cr\hline
height 13pt depth 5pt&19&
&
&
&
&
&
&
&
\DSp17\DSp&
12, $\mathbf{16_2}$&
&
&
&
&
\cr\hline
height 13pt depth 5pt&17&
&
&
&
&
&
&
\DSp16\DSp&
15, $\mathbf{18_2}$&
$\mathbf{1_2}$&
&
&
&
&
\cr\hline
height 13pt depth 5pt&15&
&
&
&
&
&
\DSp15\DSp&
17, $\mathbf{16_2}$&
$\mathbf{1_2}$&
&
&
&
&
&
\cr\hline
height 13pt depth 5pt&13&
&
&
&
&
\DSp10\DSp&
16, $\mathbf{15_2}$&
&
&
&
&
&
&
&
\cr\hline
height 13pt depth 5pt&11&
&
&
&
\DSp6\DSp&
15, $\mathbf{10_2}$&
&
&
&
&
&
&
&
&
\cr\hline
height 13pt depth 5pt&9&
&
&
\DSp3\DSp&
10, $\mathbf{6_2}$&
&
&
&
&
&
&
&
&
&
\cr\hline
height 13pt depth 5pt&7&
&
\DSp1\DSp&
7, $\mathbf{2_2}$&
&
&
&
&
&
&
&
&
&
&
\cr\hline
height 13pt depth 5pt&5&
&
2, $\mathbf{1_2}$&
&
&
&
&
&
&
&
&
&
&
&
\cr\hline
height 13pt depth 5pt&3&
\DSp1\DSp&
&
&
&
&
&
&
&
&
&
&
&
&
\cr
\dblhline
}}

\box\tablebox

%% file: knot-16n_-197566-table.tex
\def\emptyline{&&&&&&&&&&&&&&}
\def\numcolumns{16}%

\setbox\tablebox\vbox{\offinterlineskip\ialign{%
\vrule\TSp\vrule #\strut&\DSp\hfil #\DSp\vrule\TSp\vrule&
\DSp\hfil\bf #\hfil\DSp\vrule&
\DSp\hfil\bf #\hfil\DSp\vrule&
\DSp\hfil\bf #\hfil\DSp\vrule&
\DSp\hfil\bf #\hfil\DSp\vrule&
\DSp\hfil\bf #\hfil\DSp\vrule&
\DSp\hfil\bf #\hfil\DSp\vrule&
\DSp\hfil\bf #\hfil\DSp\vrule&
\DSp\hfil\bf #\hfil\DSp\vrule&
\DSp\hfil\bf #\hfil\DSp\vrule&
\DSp\hfil\bf #\hfil\DSp\vrule&
\DSp\hfil\bf #\hfil\DSp\vrule&
\DSp\hfil\bf #\hfil\DSp\vrule&
\DSp\hfil\bf #\hfil\DSp\vrule&#\TSp\vrule\cr
\dblhline
height 11pt depth 4pt&&
\rm\DSp-10\DSp&
\rm\DSp-9\DSp&
\rm\DSp-8\DSp&
\rm\DSp-7\DSp&
\rm\DSp-6\DSp&
\rm\DSp-5\DSp&
\rm\DSp-4\DSp&
\rm\DSp-3\DSp&
\rm\DSp-2\DSp&
\rm\DSp-1\DSp&
\rm\DSp0\DSp&
\rm\DSp1\DSp&
\rm\DSp2\DSp&\cr\dblhline
height 13pt depth 5pt&-3&
&
&
&
&
&
&
&
&
&
&
&
&
\DSp1\DSp&
\cr\hline
height 13pt depth 5pt&-5&
&
&
&
&
&
&
&
&
&
&
&
\DSp2\DSp&
$\mathbf{1_2}$&
\cr\hline
height 13pt depth 5pt&-7&
&
&
&
&
&
&
&
&
&
&
\DSp7\DSp&
1, $\mathbf{2_2}$&
&
\cr\hline
height 13pt depth 5pt&-9&
&
&
&
&
&
&
&
&
&
\DSp10\DSp&
3, $\mathbf{6_2}$&
&
&
\cr\hline
height 13pt depth 5pt&-11&
&
&
&
&
&
&
&
&
\DSp15\DSp&
6, $\mathbf{10_2}$&
&
&
&
\cr\hline
height 13pt depth 5pt&-13&
&
&
&
&
&
&
&
\DSp16\DSp&
10, $\mathbf{15_2}$&
&
&
&
&
\cr\hline
height 13pt depth 5pt&-15&
&
&
&
&
&
&
17, $\mathbf{1_2}$&
15, $\mathbf{16_2}$&
&
&
&
&
&
\cr\hline
height 13pt depth 5pt&-17&
&
&
&
&
&
15, $\mathbf{1_2}$&
16, $\mathbf{18_2}$&
&
&
&
&
&
&
\cr\hline
height 13pt depth 5pt&-19&
&
&
&
&
\DSp12\DSp&
17, $\mathbf{16_2}$&
&
&
&
&
&
&
&
\cr\hline
height 13pt depth 5pt&-21&
&
&
&
7, $\mathbf{1_2}$&
15, $\mathbf{12_2}$&
&
&
&
&
&
&
&
&
\cr\hline
height 13pt depth 5pt&-23&
&
&
\DSp4\DSp&
12, $\mathbf{8_2}$&
&
&
&
&
&
&
&
&
&
\cr\hline
height 13pt depth 5pt&-25&
&
\DSp1\DSp&
7, $\mathbf{3_2}\torframe{$\mathbf{1_4}$}$&
&
&
&
&
&
&
&
&
&
&
\cr\hline
height 13pt depth 5pt&-27&
&
4, $\mathbf{1_2}$&
&
&
&
&
&
&
&
&
&
&
&
\cr\hline
height 13pt depth 5pt&-29&
\DSp1\DSp&
&
&
&
&
&
&
&
&
&
&
&
&
\cr
\dblhline
}}

\box\tablebox

%% file: torus-4_-5-table-diag.pspdftex
\begin{picture}(0,0)%
\includegraphics{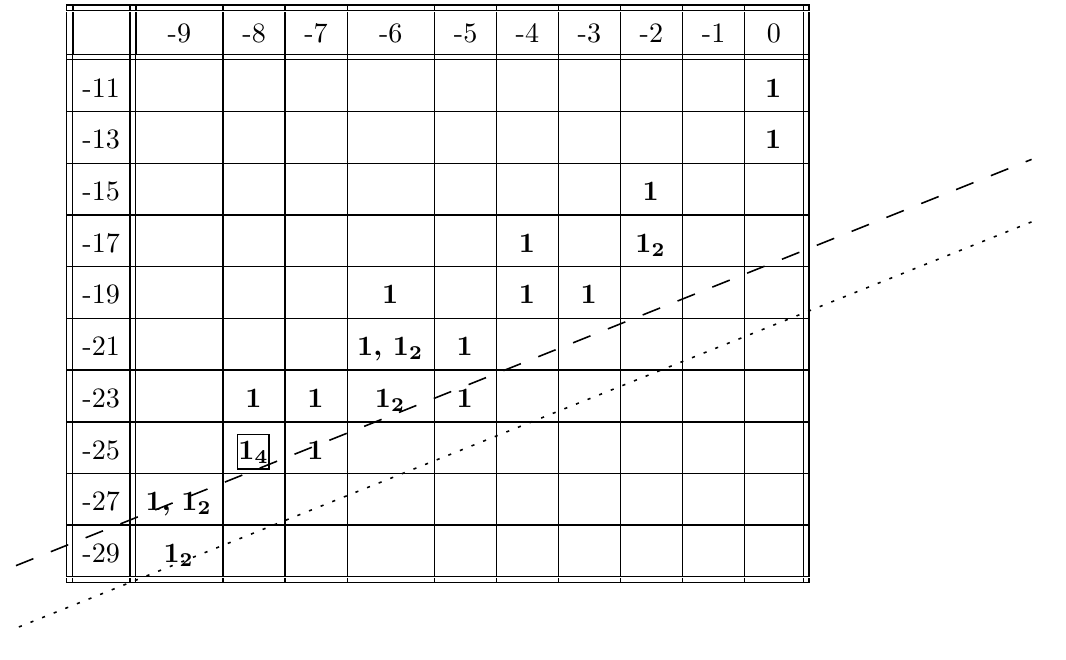}%
\end{picture}%
\setlength{\unitlength}{3947sp}%
\begingroup\makeatletter\ifx\SetFigFont\undefined%
\gdef\SetFigFont#1#2#3#4#5{%
  \reset@font\fontsize{#1}{#2pt}%
  \fontfamily{#3}\fontseries{#4}\fontshape{#5}%
  \selectfont}%
\fi\endgroup%
\begin{picture}(5102,3152)(900,-3512)
\put(5030,-1651){\rotatebox{19.0}{\makebox(0,0)[lb]{\smash{{\SetFigFont{12}{14.4}{\rmdefault}{\mddefault}{\updefault}{\color[rgb]{0,0,0}j-i=-20}%
}}}}}
\put(5026,-1341){\rotatebox{19.0}{\makebox(0,0)[lb]{\smash{{\SetFigFont{12}{14.4}{\rmdefault}{\mddefault}{\updefault}{\color[rgb]{0,0,0}j-i=-18}%
}}}}}
\end{picture}%

%% file: knot-12n_475-table.tex
\def\emptyline{&&&&&&&&&}
\def\numcolumns{11}%

\setbox\tablebox\vbox{\offinterlineskip\ialign{%
\vrule\TSp\vrule #\strut&\DSp\hfil #\DSp\vrule\TSp\vrule&
\DSp\hfil\bf #\hfil\DSp\vrule&
\DSp\hfil\bf #\hfil\DSp\vrule&
\DSp\hfil\bf #\hfil\DSp\vrule&
\DSp\hfil\bf #\hfil\DSp\vrule&
\DSp\hfil\bf #\hfil\DSp\vrule&
\DSp\hfil\bf #\hfil\DSp\vrule&
\DSp\hfil\bf #\hfil\DSp\vrule&
\DSp\hfil\bf #\hfil\DSp\vrule&#\TSp\vrule\cr
\dblhline
height 11pt depth 4pt&&
\rm\DSp0\DSp&
\rm\DSp1\DSp&
\rm\DSp2\DSp&
\rm\DSp3\DSp&
\rm\DSp4\DSp&
\rm\DSp5\DSp&
\rm\DSp6\DSp&
\rm\DSp7\DSp&\cr\dblhline
height 13pt depth 5pt&13&
&
&
&
&
&
&
&
\DSp1\DSp&
\cr\hline
height 13pt depth 5pt&11&
&
&
&
&
&
&
&
$\fam6{1}_{2}$&
\cr\hline
height 13pt depth 5pt&9&
&
&
&
&
&
\DSp1\DSp&
\DSp1\DSp&
&
\cr\hline
height 13pt depth 5pt&7&
&
&
&
&
\DSp1\DSp&
$\fam6{1}_{2}$&
&
&
\cr\hline
height 13pt depth 5pt&5&
&
&
&
&
1, $\fam6{1}_{2}$&
&
&
&
\cr\hline
height 13pt depth 5pt&3&
&
&
\DSp1\DSp&
\DSp1\DSp&
&
&
&
&
\cr\hline
height 13pt depth 5pt&1&
\DSp1\DSp&
&
$\fam6{1}_{2}$&
&
&
&
&
&
\cr\hline
height 13pt depth 5pt&-1&
\DSp1\DSp&
\DSp1\DSp&
&
&
&
&
&
&
\cr
\dblhline
}}

\box\tablebox

%% file: knot-12n_475-table-Odd.tex
\def\emptyline{&&&&&&&&&}
\def\numcolumns{11}%

\setbox\tablebox\vbox{\offinterlineskip\ialign{%
\vrule\TSp\vrule #\strut&\DSp\hfil #\DSp\vrule\TSp\vrule&
\DSp\hfil\bf #\hfil\DSp\vrule&
\DSp\hfil\bf #\hfil\DSp\vrule&
\DSp\hfil\bf #\hfil\DSp\vrule&
\DSp\hfil\bf #\hfil\DSp\vrule&
\DSp\hfil\bf #\hfil\DSp\vrule&
\DSp\hfil\bf #\hfil\DSp\vrule&
\DSp\hfil\bf #\hfil\DSp\vrule&
\DSp\hfil\bf #\hfil\DSp\vrule&#\TSp\vrule\cr
\dblhline
height 11pt depth 4pt&&
\rm\DSp0\DSp&
\rm\DSp1\DSp&
\rm\DSp2\DSp&
\rm\DSp3\DSp&
\rm\DSp4\DSp&
\rm\DSp5\DSp&
\rm\DSp6\DSp&
\rm\DSp7\DSp&\cr\dblhline
height 13pt depth 5pt&12&
&
&
&
&
&
&
&
\DSp1\DSp&
\cr\hline
height 13pt depth 5pt&10&
&
&
&
&
&
&
\DSp1\DSp&
&
\cr\hline
height 13pt depth 5pt&8&
&
&
&
&
&
\DSp1\DSp&
&
&
\cr\hline
height 13pt depth 5pt&6&
&
&
&
&
\DSp2\DSp&
&
&
&
\cr\hline
height 13pt depth 5pt&4&
&
&
&
\DSp1\DSp&
&
&
&
&
\cr\hline
height 13pt depth 5pt&2&
&
&
1, $\fam6\torframe{$\fam6{1}_{3}$}$&
&
&
&
&
&
\cr\hline
height 13pt depth 5pt&0&
&
$\fam6\torframe{$\fam6{1}_{8}$}$&
&
&
&
&
&
&
\cr\hline
height 13pt depth 5pt&-2&
$\fam6\torframe{$\fam6{1}_{3}$}$&
&
&
&
&
&
&
&
\cr
\dblhline
}}

\box\tablebox

%% file: pretzel-3-3-n3-table-Red.tex
\def\emptyline{&&&&&&&&}
\def\numcolumns{10}%

\setbox\tablebox\vbox{\offinterlineskip\ialign{%
\vrule\TSp\vrule #\strut&\DSp\hfil #\DSp\vrule\TSp\vrule&
\DSp\hfil\bf #\hfil\DSp\vrule&
\DSp\hfil\bf #\hfil\DSp\vrule&
\DSp\hfil\bf #\hfil\DSp\vrule&
\DSp\hfil\bf #\hfil\DSp\vrule&
\DSp\hfil\bf #\hfil\DSp\vrule&
\DSp\hfil\bf #\hfil\DSp\vrule&
\DSp\hfil\bf #\hfil\DSp\vrule&#\TSp\vrule\cr
\dblhline
height 11pt depth 4pt&&
\rm\DSp-6\DSp&
\rm\DSp-5\DSp&
\rm\DSp-4\DSp&
\rm\DSp-3\DSp&
\rm\DSp-2\DSp&
\rm\DSp-1\DSp&
\rm\DSp0\DSp&\cr\dblhline
height 13pt depth 5pt&0&
&
&
&
&
&
&
\DSp2\DSp&
\cr\hline
height 13pt depth 5pt&-2&
&
&
&
&
&
\DSp1\DSp&
&
\cr\hline
height 13pt depth 5pt&-4&
&
&
&
&
\DSp1\DSp&
&
&
\cr\hline
height 13pt depth 5pt&-6&
&
&
&
\DSp2\DSp&
&
&
&
\cr\hline
height 13pt depth 5pt&-8&
&
&
\DSp1\DSp&
&
&
&
&
\cr\hline
height 13pt depth 5pt&-10&
&
\DSp1\DSp&
&
&
&
&
&
\cr\hline
height 13pt depth 5pt&-12&
\DSp1\DSp&
&
&
&
&
&
&
\cr
\dblhline
}}

\box\tablebox

%% file: pretzel-3-3-n3-table-Odd.tex
\def\emptyline{&&&&&&&&}
\def\numcolumns{10}%

\setbox\tablebox\vbox{\offinterlineskip\ialign{%
\vrule\TSp\vrule #\strut&\DSp\hfil #\DSp\vrule\TSp\vrule&
\DSp\hfil\bf #\hfil\DSp\vrule&
\DSp\hfil\bf #\hfil\DSp\vrule&
\DSp\hfil\bf #\hfil\DSp\vrule&
\DSp\hfil\bf #\hfil\DSp\vrule&
\DSp\hfil\bf #\hfil\DSp\vrule&
\DSp\hfil\bf #\hfil\DSp\vrule&
\DSp\hfil\bf #\hfil\DSp\vrule&#\TSp\vrule\cr
\dblhline
height 11pt depth 4pt&&
\rm\DSp-6\DSp&
\rm\DSp-5\DSp&
\rm\DSp-4\DSp&
\rm\DSp-3\DSp&
\rm\DSp-2\DSp&
\rm\DSp-1\DSp&
\rm\DSp0\DSp&\cr\dblhline
height 13pt depth 5pt&0&
&
&
&
&
&
&
\DSp2\DSp&
\cr\hline
height 13pt depth 5pt&-2&
&
&
&
&
&
\DSp1\DSp&
$\torframe{$\mathbf{1_3}$}$&
\cr\hline
height 13pt depth 5pt&-4&
&
&
&
&
\DSp1\DSp&
&
&
\cr\hline
height 13pt depth 5pt&-6&
&
&
&
\DSp2\DSp&
&
&
&
\cr\hline
height 13pt depth 5pt&-8&
&
&
\DSp1\DSp&
&
&
&
&
\cr\hline
height 13pt depth 5pt&-10&
&
\DSp1\DSp&
&
&
&
&
&
\cr\hline
height 13pt depth 5pt&-12&
\DSp1\DSp&
&
&
&
&
&
&
\cr
\dblhline
}}

\box\tablebox

%% file: pretzel-3-4-n3-table-Red.tex
\def\emptyline{&&&&&&&&&}
\def\numcolumns{11}%

\setbox\tablebox\vbox{\offinterlineskip\ialign{%
\vrule\TSp\vrule #\strut&\DSp\hfil #\DSp\vrule\TSp\vrule&
\DSp\hfil\bf #\hfil\DSp\vrule&
\DSp\hfil\bf #\hfil\DSp\vrule&
\DSp\hfil\bf #\hfil\DSp\vrule&
\DSp\hfil\bf #\hfil\DSp\vrule&
\DSp\hfil\bf #\hfil\DSp\vrule&
\DSp\hfil\bf #\hfil\DSp\vrule&
\DSp\hfil\bf #\hfil\DSp\vrule&
\DSp\hfil\bf #\hfil\DSp\vrule&#\TSp\vrule\cr
\dblhline
height 11pt depth 4pt&&
\rm\DSp-7\DSp&
\rm\DSp-6\DSp&
\rm\DSp-5\DSp&
\rm\DSp-4\DSp&
\rm\DSp-3\DSp&
\rm\DSp-2\DSp&
\rm\DSp-1\DSp&
\rm\DSp0\DSp&\cr\dblhline
height 13pt depth 5pt&0&
&
&
&
&
&
&
&
\DSp1\DSp&
\cr\hline
height 13pt depth 5pt&-2&
&
&
&
&
&
&
\DSp1\DSp&
&
\cr\hline
height 13pt depth 5pt&-4&
&
&
&
&
&
\DSp1\DSp&
&
&
\cr\hline
height 13pt depth 5pt&-6&
&
&
&
&
\DSp1\DSp&
&
&
&
\cr\hline
height 13pt depth 5pt&-8&
&
&
&
\DSp2\DSp&
&
&
&
&
\cr\hline
height 13pt depth 5pt&-10&
&
&
\DSp1\DSp&
&
&
&
&
&
\cr\hline
height 13pt depth 5pt&-12&
&
\DSp1\DSp&
&
&
&
&
&
&
\cr\hline
height 13pt depth 5pt&-14&
\DSp1\DSp&
&
&
&
&
&
&
&
\cr
\dblhline
}}

\box\tablebox

%% file: pretzel-3-4-n3-table-Odd.tex
\def\emptyline{&&&&&&&&&}
\def\numcolumns{11}%

\setbox\tablebox\vbox{\offinterlineskip\ialign{%
\vrule\TSp\vrule #\strut&\DSp\hfil #\DSp\vrule\TSp\vrule&
\DSp\hfil\bf #\hfil\DSp\vrule&
\DSp\hfil\bf #\hfil\DSp\vrule&
\DSp\hfil\bf #\hfil\DSp\vrule&
\DSp\hfil\bf #\hfil\DSp\vrule&
\DSp\hfil\bf #\hfil\DSp\vrule&
\DSp\hfil\bf #\hfil\DSp\vrule&
\DSp\hfil\bf #\hfil\DSp\vrule&
\DSp\hfil\bf #\hfil\DSp\vrule&#\TSp\vrule\cr
\dblhline
height 11pt depth 4pt&&
\rm\DSp-7\DSp&
\rm\DSp-6\DSp&
\rm\DSp-5\DSp&
\rm\DSp-4\DSp&
\rm\DSp-3\DSp&
\rm\DSp-2\DSp&
\rm\DSp-1\DSp&
\rm\DSp0\DSp&\cr\dblhline
height 13pt depth 5pt&0&
&
&
&
&
&
&
&
\DSp1\DSp&
\cr\hline
height 13pt depth 5pt&-2&
&
&
&
&
&
&
\DSp1\DSp&
&
\cr\hline
height 13pt depth 5pt&-4&
&
&
&
&
&
\DSp1\DSp&
$\torframe{$\mathbf{1_3}$}$&
&
\cr\hline
height 13pt depth 5pt&-6&
&
&
&
&
\DSp1\DSp&
&
&
&
\cr\hline
height 13pt depth 5pt&-8&
&
&
&
\DSp2\DSp&
&
&
&
&
\cr\hline
height 13pt depth 5pt&-10&
&
&
\DSp1\DSp&
&
&
&
&
&
\cr\hline
height 13pt depth 5pt&-12&
&
\DSp1\DSp&
&
&
&
&
&
&
\cr\hline
height 13pt depth 5pt&-14&
\DSp1\DSp&
&
&
&
&
&
&
&
\cr
\dblhline
}}

\box\tablebox

%% file: pretzel-link.pspdftex
\begin{picture}(0,0)%
\includegraphics{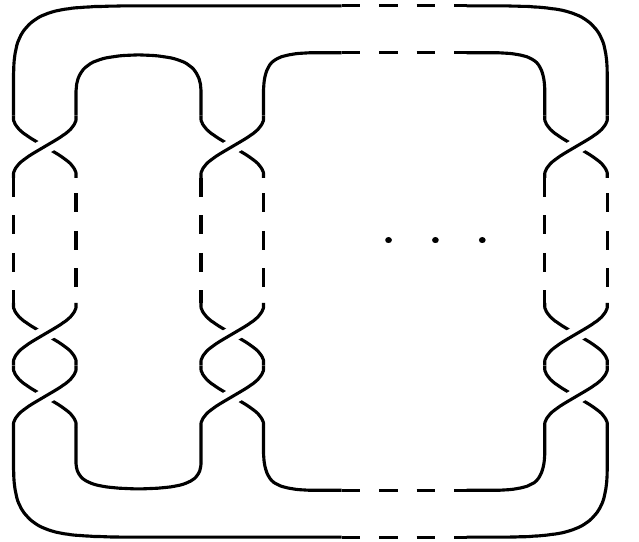}%
\end{picture}%
\setlength{\unitlength}{3947sp}%
\begingroup\makeatletter\ifx\SetFigFont\undefined%
\gdef\SetFigFont#1#2#3#4#5{%
  \reset@font\fontsize{#1}{#2pt}%
  \fontfamily{#3}\fontseries{#4}\fontshape{#5}%
  \selectfont}%
\fi\endgroup%
\begin{picture}(3054,2594)(536,-2333)
\put(1876,-1036){\rotatebox{90.0}{\makebox(0,0)[b]{\smash{{\SetFigFont{10}{12.0}{\rmdefault}{\mddefault}{\updefault}{\color[rgb]{0,0,0}$\underbrace{\hbox to1.3truein{\hfill}}_{\displaystyle p_2\hbox{ right twists}}$}%
}}}}}
\put(976,-1036){\rotatebox{90.0}{\makebox(0,0)[b]{\smash{{\SetFigFont{10}{12.0}{\rmdefault}{\mddefault}{\updefault}{\color[rgb]{0,0,0}$\underbrace{\hbox to1.3truein{\hfill}}_{\displaystyle p_1\hbox{ right twists}}$}%
}}}}}
\put(3526,-1036){\rotatebox{90.0}{\makebox(0,0)[b]{\smash{{\SetFigFont{10}{12.0}{\rmdefault}{\mddefault}{\updefault}{\color[rgb]{0,0,0}$\underbrace{\hbox to1.3truein{\hfill}}_{\displaystyle p_n\hbox{ right twists}}$}%
}}}}}
\end{picture}%

%% file: pretzel-3_4_n3.pspdftex
\begin{picture}(0,0)%
\includegraphics{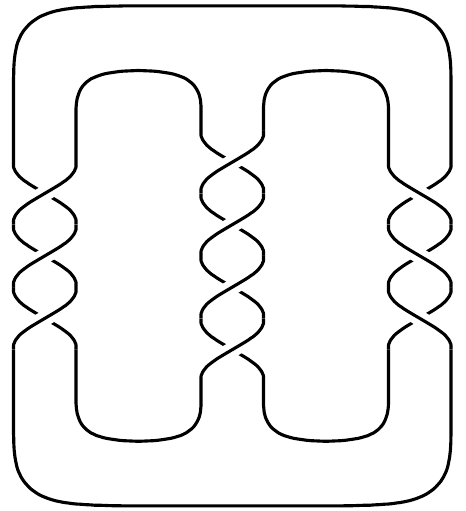}%
\end{picture}%
\setlength{\unitlength}{3947sp}%
\begingroup\makeatletter\ifx\SetFigFont\undefined%
\gdef\SetFigFont#1#2#3#4#5{%
  \reset@font\fontsize{#1}{#2pt}%
  \fontfamily{#3}\fontseries{#4}\fontshape{#5}%
  \selectfont}%
\fi\endgroup%
\begin{picture}(2230,2444)(536,-2183)
\end{picture}%

%% file: KhSurvey.bbl
\begin{thebibliography}{mmmm}
\raggedright

\bibitem[AKh]{Kh-Asaeda}
M.~Asaeda and M.~Khovanov, {\sl Notes on link homology}, in {\sl Low
   dimensional topology}, 139--195, IAS/Park City Math. Ser., 15, Amer. Math.
   Soc., Providence, RI, 2009; arXiv:0804.1279.

\bibitem[BN1]{BN-first}
D.~Bar-Natan, {\sl On Khovanov's categorification of the Jones polynomial},
   Alg. Geom. Top., {\bf 2} (2002) 337--370; arXiv:math.QA/0201043.

\bibitem[BN2]{BN-fast}
D.~Bar-Natan, {\sl Fast Khovanov Homology Computations}, J. Knot Th. and
    Ramif. {\bf 16} (2007), no. 3, 243--255; arXiv:math.GT/0606318.

\bibitem[BNG]{katlas-program}
D.~Bar-Natan and J.~Green, {\tt JavaKh} --- a fast program for computing
    Khovanov homology, part of the {\tt KnotTheory`} Mathematica Package,
    {\tt http://katlas.math.utoronto.ca/wiki/Khovanov\char"5F Homology}

\bibitem[BW]{Beliakova-colored}
A.~Beliakova and S.~Wehrli, {\sl Categorification of the colored Jones
    polynomial and Rasmussen invariant of links},  Canad. J. Math. {\bf 60}
    (2008), no. 6, 1240--1266; arXiv:math.GT/0510382.

\bibitem[B]{Bloom-mutation}
J.~Bloom, {\sl Odd Khovanov homology is mutation invariant}, Math. Res. Lett.
   {\bf 17} (2010), no. 1, 1--10; arXiv:0903.3746.

\bibitem[ChK]{Kofman-Co-QA}
A.~Champanerkar and I.~Kofman, {\sl Twisting quasi-alternating links},
   Proc. Amer. Math. Soc. {\bf 137} (2009), 2451--2458; arXiv:0712.2590.

\bibitem[DGShT]{DGShT}
N.~Dunfield, S.~Garoufalidis, A.~Shumakovitch, and M.~Thistlethwaite,
   {\sl Behavior of knot invariants under genus 2 mutation},
   New York J. Math. {\bf 16} (2010), 99--123; arXiv:math.GT/0607258.

\bibitem[F]{Freedman}
M.~Freedman, {\sl A surgery sequence in dimension four; the relations
    with knot concordance}, Invent. Math. {\bf 68} (1982), no. 2, 195-226.

\bibitem[GS]{Gompf}
R.~Gompf and A.~Stipsicz, {\sl $4$-manifolds and Kirby calculus}, Graduate
    Studies in Mathematics {\bf 20}, American Mathematical Society,
    Providence, RI, 1999.

\bibitem[Gr]{Greene-QA}
J.~Greene, {\sl Homologically thin, non-quasi-alternating links},
   Math. Res. Lett. {\bf 17} (2010), no. 1, 39--49; arXiv:0906.2222.

\bibitem[HTh]{Knotscape}
J.~Hoste and M.~Thistlethwaite, {\tt Knotscape} --- a program for
    studying knot theory and providing convenient access to tables of
    knots, {\tt http://www.math.utk.edu/\~{}morwen/knotscape.html}

\bibitem[J]{Jones}
V.~Jones, {\sl A polynomial invariant for knots via von Neumann algebras},
   Bull. Amer. Math. Soc. {\bf 12} (1985), 103--111.

\bibitem[K]{Kauffman-bracket}
L.~Kauffman, {\sl State models and the Jones polynomial}, Topology {\bf 26}
   (1987), no. 3, 395--407.

\bibitem[Kh1]{Kh-Jones}
M.~Khovanov, {\sl A categorification of the Jones polynomial},
    Duke Math. J. {\bf 101} (2000), no.~3, 359--426; arXiv:math.QA/9908171.

\bibitem[Kh2]{Kh-patterns}
M.~Khovanov, {\sl Patterns in knot cohomology I}, Experiment. Math.  {\bf 12}
    (2003), no. 3, 365--374; arXiv:math.QA/0201306.

\bibitem[Kh3]{Kh-colored}
M.~Khovanov, {\sl  Categorifications of the colored Jones polynomial},
    J. Knot Th. Ramif. {\bf 14} (2005), no. 1, 111--130;
    arXiv:math.QA/0302060.

\bibitem[Kh4]{Kh-sl3}
M.~Khovanov, {\sl $\slf(3)$ link homology I}, Algebr. Geom. Topol. {\bf 4}
    (2004), 1045--1081; arXiv:math.QA/0304375.

\bibitem[Kh5]{Kh-Frobenius}
M.~Khovanov, {\sl Link homology and Frobenius extensions},
   Fundamenta Mathematicae, {\bf 190} (2006), 179--190; arXiv:math.QA/0411447.

\bibitem[Kh6]{Kh-ICM}
M.~Khovanov, {\sl Link homology and categorification}, ICM--2006, Madrid,
   Vol. II, 989--999, Eur. Math. Soc., Z\"urich, 2006; arXiv:math/0605339.

\bibitem[KhR1]{KR-sln}
M.~Khovanov and L.~Rozansky, {\sl Matrix factorizations and link homology},
   Fund. Math. {\bf 199} (2008), no. 1, 1--91; arXiv:math.QA/0401268.

\bibitem[KhR2]{KR-HOMFLY}
M.~Khovanov and L.~Rozansky, {\sl Matrix factorizations and link homology II},
   Geom. Topol. {\bf 12} (2008), no. 3, 1387--1425; arXiv:math.QA/0505056.

\bibitem[KhR3]{KR-SO2N}
M.~Khovanov and L.~Rozansky, {\sl Virtual crossings, convolutions and a
   categorification of the $SO(2N)$ Kauffman polynomial}, J. G\"okova Geom.
   Topol. GGT {\bf 1} (2007), 116--214; arXiv:math/0701333.

\bibitem[KM1]{Kronheimer-Mrowka}
   P.~Kronheimer and T.~Mrowka, {\sl Gauge Theory for Embedded Surfaces I},
   Topology {\bf 32} (1993), 773--826.

\bibitem[KM2]{Kronheimer-Mrowka-unknot}
   P.~Kronheimer and T.~Mrowka, {\sl Khovanov homology is an unknot-detector},
   arXiv:1005.4346.

\bibitem[L]{Lee}
E.~S.~Lee, {\sl An endomorphism of the Khovanov invariant}, Adv. Math.
   {\bf 197} (2005), no. 2, 554--586; arXiv:math.GT/0210213.

\bibitem[MO]{QA-links}
C.~Manolescu, P.~Ozsv\'ath, {\sl On the Khovanov and knot Floer
   homologies of quasi-alternating links} in  Proc. of G\"okova
   Geometry-Topology Conference 2007, 60--81, G\"okova Geometry/Topology
   Conference (GGT), G\"okova, 2008; arXiv:0708.3249.

\bibitem[Ng]{Ng-TB-bound}
L.~Ng, {\sl A Legendrian Thurston-Bennequin bound from Khovanov homology},
   Algebr. Geom. Topol. {\bf 5} (2005) 1637--1653; arXiv:math.GT/0508649.

\bibitem[ORS]{Khovanov-odd}
P.~Ozsv\'ath, J.~Rasmussen, and Z.~Szab\'o, {\sl Odd Khovanov homology},
   arXiv:0710.4300.

\bibitem[OS1]{OS-HF}
P.~Ozsv\'ath and Z.~Szab\'o, {\sl Holomorphic disks and topological invariants
   for closed three-manifolds}, Ann. of Math. {\bf 159} (2004) 1027--1158;
   arXiv:math.SG/0101206.

\bibitem[OS2]{OS-knots}
P.~Ozsv\'ath and Z.~Szab\'o, {\sl Holomorphic disks and knot invariants},
   Adv. Math. {\bf 186} (2004), no. 1, 58--116;  arXiv:math.GT/0209056.

\bibitem[OS3]{OS-spectral}
P.~Ozsv\'ath and Z.~Szab\'o, {\sl On the Heegaard Floer homology of branched
   double-covers}, Adv. Math. {\bf 194} (2005), no. 1, 1--33;
   arXiv:math.GT/0309170.

\bibitem[P]{Plamenevskaya}
O. Plamenevskaya, {\sl Transverse knots and Khovanov homology},
    Math. Res. Lett. {\bf 13} (2006), no. 4, 571--586; arXiv:math.GT/0412184.

\bibitem[Ra1]{Jake-Floer}
J.~Rasmussen, {\sl Floer homology and knot complements}, Ph.D. Thesis,
   Harvard U.; arXiv:math.GT/0306378.

\bibitem[Ra2]{Jake-Milnor}
J.~Rasmussen, {\sl Khovanov homology and the slice genus}, to appear in
   Invent. Math.; arXiv:math.GT/0402131.

\bibitem[Ra3]{Jake-comparison}
J.~Rasmussen, {\sl Knot polynomials and knot homologies}, Geometry and
   Topology of Manifolds (Boden et al eds.), Fields Institute Communications
   {\bf 47} (2005) 261--280, AMS; arXiv:math.GT/0504045.

\bibitem[Ro]{Rolfsen-book}
D.~Rolfsen, {\sl Knots and Links}, Publish or Perish, Mathematics
    Lecture Series 7, Wilmington 1976.

\bibitem[Ru]{Rudolph-quasipositivity}
L.~Rudolph, {\sl Quasipositivity as an obstruction to sliceness},
    Bull Amer. Math. Soc. (N.S.) {\bf 29} (1993), no. 1, 51--59;
    arXiv:math.GT/9307233.

\bibitem[Sh1]{Sh-KhoHo}
A.~Shumakovitch, {\tt KhoHo} --- a program for computing and studying
    Khovanov homology, {\tt http://www.geometrie.ch/KhoHo}

\bibitem[Sh2]{Sh-torsion}
A.~Shumakovitch, {\sl Torsion of the Khovanov homology},
   arXiv:math.GT/0405474; to appear in Fund. Math.

\bibitem[Sh3]{Sh-Rasmussen}
A.~Shumakovitch, {\sl Rasmussen invariant, Slice-Bennequin inequality, and
   sliceness of knots}, J. Knot Th. and Ramif., {\bf 16} (2007), no. 10,
   1403--1412; arXiv:math.GT/0411643.

\bibitem[Sh4]{Sh-OddKhovanov}
A.~Shumakovitch, {\sl Patterns in odd Khovanov homology}, to appear in
   J. Knot Th. and Ramif.

\bibitem[T]{Turner-Char2}
P.~Turner, {\sl Calculating Bar-Natan's characteristic two Khovanov homology},
   J. Knot Th. Ramif. {\bf 15} (2006), no. 10, 1335--1356;
   arXiv:math.GT/0411225.

\bibitem[V]{Viro-defs}
O. Viro, {\sl Khovanov homology, its definitions and ramifications}, Fund.
    Math. {\bf 184} (2004), 317--342; arXiv:math.GT/0202199.

\bibitem[W1]{Wehrli-mutation1}
S.~Wehrli, {\sl Khovanov Homology and Conway Mutation},
   arXiv:math/0301312.

\bibitem[W2]{Wehrli-mutation2}
S.~Wehrli, {\sl Mutation invariance of Khovanov homology over $\mathbb{F}_2$},
    Quantum Topol. {\bf 1} (2010), no. 2, 111--128; arXiv:0904.3401.

\bibitem[Wu]{Wu-slN}
H.~Wu, {\sl A colored $\slf(N)$-homology for links in $S^3$}, arXiv:0907.0695.

\bibitem[Y]{Yonezawa}
Y.~Yonezawa, {\sl Quantum $(\slf_n,\land V_n)$ link invariant and matrix
   factorizations}, arXiv:0906.0220.

\end{thebibliography}
